\documentclass[11pt]{article}

\usepackage[pagebackref,colorlinks=true,pdfpagemode=none,urlcolor=blue,
linkcolor=blue,citecolor=blue]{hyperref}

\usepackage{amsmath,amsfonts,amssymb,amsthm}
\usepackage{mathrsfs}
\usepackage{bm}
\usepackage{color}
\usepackage{algorithm}
\usepackage{algorithmic}
\usepackage{graphicx}
\usepackage{subfigure}

\usepackage[utf8]{inputenc}            
\usepackage[T1]{fontenc}

\usepackage{common_style}
\usepackage{paper_style}

\title{Time-domain multiscale shape identification in electro-sensing \thanks{\footnotesize
    This work was supported by ERC Advanced Grant Project MULTIMOD--267184.}}  \author{Habib Ammari\thanks{\footnotesize Department of
    Mathematics and Applications, Ecole Normale Sup\'erieure, 45 Rue d'Ulm, 75005 Paris, France
    (habib.ammari@ens.fr, han.wang@ens.fr).} \and Han Wang\footnotemark[2]}

\graphicspath{{figures/}}

\begin{document}

\maketitle

\begin{abstract}
  This paper presents premier and innovative time-domain multi-scale method for shape
  identification in electro-sensing using pulse-type signals. The method is based on
  transform-invariant shape descriptors computed from filtered polarization tensors at
  multi-scales. The proposed algorithm enjoys a remarkable noise robustness even with
  far-field measurements at very limited angle of view. It opens a door for pulsed imaging
  using echolocation and induction data.
\end{abstract}

\bigskip

\noindent {\footnotesize Mathematics Subject Classification
  (MSC2000): 35R30, 35B30}


\section{Introduction}
Weakly electric fish orient themselves at night in complete darkness by employing their
active electrolocation system. They generate a stable, high-frequency, weak electric field
and perceive the transdermal potential modulations caused by a nearby target with
different electromagnetic properties than the surrounding water \cite{bastian,
  lissmann1958mechanism, moller1995, von1993electric}. Depending on the waveform of the
source (\ie the electric organ discharge) which is a result of the evolution and is
adapted to the habitat, weakly electric fish can be classified into the wave-type and the
pulse-type \cite{bastian}.  The first emit a sinusoidal-like signal while the second emit
brief pulses.  Both types of fish have to solve the electro-sensing problem: locate the
target and identify its shape and electromagnetic parameters given the current
distribution over the skin. Due to the fundamental ill-posedness character of this imaging
problem, it is very intriguing to see how much information weakly electric fish are able
to recover \cite{budelli2000electric, mciver3, maciver2001computational, maciver2001prey,
  mciverreview, mciver, expt, von1999active, gerhard, vonreview, von2007distance}.

A solution to the electric-sensing problem relies on differential imaging, {\it i.e.}, by
forming an image from the perturbations of the field due to targets, and physics-based
classification. The electric field due to the target is a complicated highly nonlinear function
of its shape, electromagnetic parameters, and distance from the fish. Differential imaging helps us to
understand analytically the electric sense of the weakly electric fish.

In a recent paper \cite{ammari_modeling_2013} a mathematical model of the fish has been
established. Based on this model, a multifrequency shape recognition algorithm for wave-type
fish has been proposed in \cite{ammari_shape_2014}. To summarize, the wave-type fish would
first locate the target using a specific frequency-space location search algorithm. Then it
could extract, from the perturbations of the electric field, the polarization tensors of the
target at multiple frequencies. The material parameters of the target can be computed from
these extracted features. Finally, the fish might classify a target by comparing normalized
singular values of the polarization tensors with those of a set of learned shapes. These
geometric features extracted from the data are invariant under rigid motions and scaling of the
target and therefore, they yield shape descriptors which allow the comparison and
identification of the target in a dictionary of shapes.  

In this paper, we study the problem of shape identification using pulse-type
signals. Compared to previous investigations on wave-type electro-sensing, the present
model is more complex and appears to be more realistic since shape identification performs
much better even with a limited-view aspect and highly noisy data.

The overall procedure of electro-sensing is similar to the wave-type electro-sensing
described above. However, unlike the wave-type electro-sensing where the solution of the
forward problem in the frequency domain is separable and can be treated independently for
each frequency, the shape identification problem using pulse-type signals has to be
treated directly in the time domain hence is more challenging. On the other hand, the
pulse-type signal contains more information from a frequency point of view and is expected
to give a better performance than wave-type signals in shape identification.

The paper is organized as follows. We first introduce  some notation. Then in
section \ref{sec:electro-sens-model} we establish a simplified electro-sensing model in the time domain. Section \ref{sec:repr-solut}
gives the representation of the solution. Section \ref{sec:time-depend-polar} is to formulate an asymptotic
expansion of the perturbed field which allows the reconstruction of the filtered
generalized polarization tensors (GPTs) from data. 

Based on the polarization tensor, a time domain multi-scale shape descriptor is introduced
in section \ref{sec:shape-ident-with} and its performance is analyzed through numerical
experiments in section \ref{sec:numer-exper}. The paper ends with a few concluding
remarks.

Throughout this paper, we denote by $\Gamma$ the fundamental solution of the Laplacian in $\R^d$ with $d=2,3$, which satisfies $\Delta
\Gamma =\delta_0$ (where $\delta_0$ is the Dirac function at the origin) and is given by
\begin{align}
  \label{eq:def_green}
  \Gamma(x) := 
  \begin{cases}
    \frac{1}{2\pi} \log\abs{x}, & d=2,\\
    -\frac{1}{4\pi} \frac 1{\abs{x}}, & d=3.
  \end{cases}
\end{align}

For a Banach space $X$ equipped with the norm $\normx{\cdot}$, we define the Schwartz
space $\SzRX$ as follows
\begin{align}
  \label{eq:Schwartz_space}
  \SzRX := \Set{\phi:\R\to X \text{ is } \Cinf, \text{ and } \pab(\phi)<\infty, \forall
    a,b\in\NN},
\end{align}
where the semi norms $\pab$ for $a,b\in\NN$ are defined as
\begin{align}
  \label{eq:pab_norms}
  \pab(\phi) = \sup_{t\in\R} \,\abs{t}^a \normx{\phi^{(b)}(t)}.
\end{align}
We denote by $\SzRXp$ the space of tempered distributions. The Fourier transform defined
as $\displaystyle{\hat\phi(\omega) = \int_\R \phi(t) e^{-it\omega}\, dt}$ for a function
of $\SzRX$ is always carried out on the time variable $t$, and for a distribution of
$\SzRXp$ it is defined by the duality. In both cases the Fourier transform is a
homeomorphism on the corresponding space.

We denote $\LtIX$ the space of square integrable functions $\phi:\Intv\rightarrow X$, and equip
it with the norm
\begin{align}
  \label{eq:time_X_space_L2}
  \norm{\phi}_\LtIX := \Paren{\int_\Intv\normx{\phi(t)}^2\, dt}^{1/2},
\end{align}
Similarly, we denote $H^1(\Intv; X)$ the space of all $\phi\in\LtIX$ such that the weak
derivative $\p_t \phi\in\LtIX$ and equip it with the norm
\begin{align}
  \label{eq:time_X_space_H1}
  \norm{\phi}_\HoIX := \Paren{\norm{\phi}_\LtIX^2 + \norm{\phi'}_\LtIX^2}^{1/2}.  
\end{align}
Throughout the paper we will write interchangeably $\phi'(t,x)$ and $\p_t \phi(t,x)$ for
the derivative in the time variable $t$ (similarly $\hat\phi'(\omega,x)$ and
$\p_\omega\hat\phi(\omega,x)$ for the derivative in the frequency variable $\omega$ for
the Fourier transform of $\phi$). We call a function $\phi$ \emph{causal} if $\phi(t) = 0$
for $t<0$. Particularly, $\phi\in\SzRX$ being causal implies $\phi^{(k)}(0)=0$ for any
$k\geq 0$.

\section{Electro-sensing model}
\label{sec:electro-sens-model}

We consider in this paper the electro-sensing problem in the free space $\R^d$ with point
sources and receivers, which is easier to analyse compared to the complete  model of fish
established in \cite{ammari_modeling_2013,ammari_shape_2014}. Before proceeding to the
results of existence and uniqueness of the solution as well as its representation, we want
to insist on the fact that the same type of results can be established in a similar way
for the model of \cite{ammari_modeling_2013}, in particular the shape identification
algorithm discussed in section \ref{sec:shape-ident-with} remains unchanged and applies to
any model as long as the same feature is extracted.

A target $D$ is an open bounded set in $\R^d, d=2,3$, of class $\mathcal{C}^{1,\alpha}$,
$0<\alpha<1$, and we can represent it as $D=z+\delta B$, where $B$ is the reference domain of size
$1$ containing the origin, $\delta \ll 1$ is the characteristic size of $D$, and $z$ is its
location. The characteristic function of $D$ is denoted by $\chid$, and its constant
conductivity and permittivity are denoted by $\sigma$ and $\e$ respectively with $\sgmd>0,
\vepd>0$. The conductivity and permittivity distributions of the whole space are piecewise constant:
\begin{align}
  \label{eq:cond_pmtt_def}
  \sigma(x) &= \sgmo + (\sgmd-\sgmo)\chid(x), \
  \text{ and } \vep(x) = \vepo + (\vepd-\vepo)\chid(x)
\end{align}
where $\sgmo>0, \vepo\geq 0$ and $\sgmo\neq \sgmd, \vepo\neq\vepd$ are the background values and $\chid$ is the characteristic function of $D$.

\subsection{Governing equation for the voltage potential}
\label{sec:potential-field}

Under the electro-quasi-static (or EQS) approximation of the Maxwell's system, the
electric field reads $E(t,x)=\nabla u(t,x)$, where $u$ is  the voltage potential, and the magnetic field $H$ satisfies
\begin{align}
  \label{eq:Amperes_law}
  \nabla\times H(t,x) = \vep(x)\p_t E(t,x) + J(t,x),
\end{align}
where the current density $J(t,x)=\sigma(x)E(t,x)+J_s(t,x)$, \ie the sum of induction and source
current. Let $f(t,x):=-\nabla. J_s(t,x)$ be the source. Taking the divergence of
\eqref{eq:Amperes_law}, we get
\begin{align*}
  \nabla . (\sigma(x)E(t,x) + \vep(x)\p_t E(t,x)) = -\nabla. J_s(t,x) = f(t,x),
\end{align*}
or in terms of  $u$,
\begin{align}
  \label{eq:Amperes_law_u}
  \nabla.(\sigma(x) + \vep(x)\p_t)\nabla u(t,x) = f(t,x).
\end{align}

We complete \eqref{eq:Amperes_law_u} by a decay condition at
infinity 
as well as an initial condition at $t=0$, and obtain the governing equation of the voltage potential
\begin{equation}
  \label{eq:governing_eq}
  \left\{
    \begin{alignedat}{2}
      \nabla.(\sigma(x) + \vep(x)\p_t)\nabla u(t,x) &= f(t,x) \ \text{ in } \R_+\times \DuDc
      \ ,
      \\
      \abs{u(t,x)} &= O(\abs{x}^{1-d}) \ \text{ as } \abs{x}\rightarrow +\infty,
      t\in \R_+ \ ,\\
      u(0,x) &= u_0(x) \ \text{ in } \DuDc \ .
    \end{alignedat}
  \right.
\end{equation}

For the eletro-sensing problem in water we typically set for the surrounding water $\sgmo = 1$
and $\vepo = 0$. Furthermore, we suppose there is no potential at the initial state and let $\Dc = \DuDc \setminus \overline D$. Under
these settings, it is easy to see that \eqref{eq:governing_eq} can be rewritten as the
following transmission problem:
\begin{equation}
  \label{eq:governing_eq_trans}
  \def\vepd{\vep}
  \def\sgmd{\sigma}
  \left\{
    \begin{alignedat}{2}
      \vepd \Delta u'(t) + \sgmd \Delta u(t) &= 0 \ \text{ in } \R_+ \times D \ , \\
      \Delta u(t) &= f(t) \ \text{ in } \R_+\times \Dc \ , \\
      {u(t)} \Big\vert_- &= {u(t)} \Big\vert_+ \ \text{ on } \R_+ \times \p D \ ,\\
      \vepd \ddn{u'(t)} \Big\vert_- + \sgmd \ddn{u(t)} \Big\vert_- &= \ddn{u(t)} \Big\vert_+ \
      \text{ on } \R_+ \times \p D \ ,\\
      \abs{u(t,x)} &= O(\abs{x}^{1-d}) \ \text{ as } \abs{x}\rightarrow +\infty, t\in \R_+ \
      ,\\
      u(0, x) &= 0 \ \text{ in } \DuDc \ .
    \end{alignedat}  
  \right.
\end{equation}

\subsubsection{Uniqueness of a solution to the governing equation}
\label{sec:uniqueness-solution}
We define the Banach spaces 
\begin{align}
  \label{eq:def_Xsp_Xspp}
  \Xsp=\Hloc(\DuDc), \ \Xsp'=H^{-1}(\DuDc),
\end{align}
and consider \eqref{eq:governing_eq} in $\HoRpX$ with the initial condition $u_0\in \Xsp$, and
the source term $f\in\HoRpXp$.

\begin{lem}\label{lem:uniqueness_of_gov_eq}
  If the solution to \eqref{eq:governing_eq} fulfills $u\in\HoRpX$, then it is unique.
\end{lem}
\begin{proof}
  We introduce two bilinear forms
  \begin{align}
    \label{eq:bilinear_a1_a2}
    a_1(u,v) = \int_{\R^d} \vep(x)\nabla u(x) . \nabla v(x) dx, \   a_2(u,v) = \int_{\R^d}
    \sigma(x)\nabla u(x) . \nabla v(x) dx.
  \end{align}

  Let $u_1, u_2$ be two solutions to \eqref{eq:governing_eq} in  $\HoRpX$. Then, their
  difference $w=u_1-u_2\in \HoRpX$ must solve 
  \begin{equation}
    \label{eq:governing_eq_w}
    \left\{
      \begin{alignedat}{2}
        \nabla.(\sigma(x) + \vep(x)\p_t)\nabla w(t,x) &= 0, \ \text{ in } \R_+ \times \DuDc \
        ,
        \\
        \abs{w(t,x)} &= O(\abs{x}^{1-d}), \ \text{ as } \abs{x}\rightarrow +\infty,
        t\in \R_+ \ ,\\
        w(0,x) &= 0, \ \text{ in } \DuDc \ .
      \end{alignedat}
    \right.
  \end{equation}
  Multiplying the first line by a test function $\vphi\in \CDD$ and integrating by parts in $\R^d$
  yield:
  \begin{align*}
    a_1(w(t), \vphi) + a_2(w'(t), \vphi) = 0, \ \text{ for a.e. }\  t\in \R_+ .
  \end{align*}
  which implies, by the density of $\CDD$ in $H^1(\DuDc)$,
  \begin{align*}
    a_1(w'(t), w(t)) + a_2(w(t), w(t)) = 0, \ \text{ for } t\in \R_+ \ a.e.
  \end{align*}
  
  For any $T>0$, integrating the expression above on $[0,T]$ and using the initial condition
  $\nabla w(0,x)=0$ gives
  \begin{align*}
    \hlf\int_{\R^d} \vep(x) \abs{\nabla w'(T)}^2 \,dx + \int_{\R^d} \sigma(x) \int_0^T
    \abs{\nabla w(t)}^2 \,dt\,dx = 0,
  \end{align*}
  which means, since $\sigma(x)>0$ and $\vep(x)\geq 0$, that $\abs{\nabla w(t,x)}^2=0$ in
  $[0,T]\times \R^d$. Since $T>0$ is arbitrary, combining this with the decay condition in
  \eqref{eq:governing_eq_w} implies $u_1(t,x)=u_2(t,x)$ a.e. in $\R_+\times\R^d$.

\end{proof}

\subsection{Electric organ and pulse-type signals}
\label{sec:electr-organ-wavef}
The time-varying source current $f$ emitted by the fish can be modeled as
\begin{align}
  \label{eq:source_f}
  f(t,x)= h(t) \tf(x),
\end{align}
with $h$ being the shape form (\ie the time profile) of the source. $\tf$ is a function modeling
the electric organ:
\begin{align}
  \label{eq:source_fx}
  \tf(x) = \sum_{j=1}^p a_j \delta_0(x-\xsj) \  \text{ with } \xsj\in \Dc,
\end{align}
where $\xsj\in\R^d, j=1\ldots p$ are the point sources and characterize the spatial distribution of the electric
organ, and $a_j$ fulfills the neutrality condition:
\begin{align}
  \label{eq:neutrality}
  \sum_{j=1}^p a_j = 0,
\end{align}
which insures the decay behavior $ \abs{u(t,x)} = O(\abs{x}^{1-d})$ at infinity. 

We refer the reader to \cite{ammari_modeling_2013} for more details on the modeling of the
electric organ and \cite{ammari_shape_2014} for the electrolocation using wave-type signals.
Throughout this paper, we will consider the pulse shape form $h$ under
the assumption
\begin{align}
  \label{eq:h_assump1}
  h \text{ is } causal \text{ and } h\in\Sz(\R),
\end{align}
where $\Sz(\R)$ is the classical Schwartz space. As a simple consequence it holds
$h^{(k)}(0)=0$ for any $k\geq 0$. 

It is worth emphasizing that causality is important issue because of physical
considerations. Throughout this paper, we will carefully check that the solution to the
electro-sensing problem is causal.







\section{Representation of solution}
\label{sec:repr-solut}

We introduce in this section an integral representation of the solution of the problem
\eqref{eq:governing_eq_trans}. The following notation will be used in this section. Let
\begin{align}
  \label{eq:some_constants}
  \kappao := \sigma+i\vep\omega, \ \lambdao := \frac{\kappao + 1}{2(\kappao - 1)} , \ \lambda
  := \frac{\sigma+1}{2(\sigma-1)}, \ \alpha := \frac{\vep}{\sigma-1}.
\end{align} 
We call $\kappao$ the admittivity. 



\subsection{Layer potentials}
\label{sec:layer-potentials}
Let the single layer potential of a density $\phi \in \LtpD$ be defined by
\begin{align} 
  \label{defs}
  \Sglf D \phi (x) &:= \int_{\p D} \Gamma(x-y) \phi (y) \, d \sigma(y), \quad x \in
  \mathbb{R}^d.
\end{align}
It is well-known that $\Sglf D \phi$ is harmonic on $\R^d\setminus \p D$. Let
Neumann-Poincar\'e operator $\Kstar D$ on $\LtpD$ be given by
\begin{equation} 
  \label{defk}
  \Kstarf D \phi (x) := \int_{\p D} \frac
  {\partial \Gamma}{\partial {\nu(x)}}(x-y) \phi (y)\,ds(y), \quad
  \phi \in \LtpD.
\end{equation}
Then we have the jump formula for the single layer potential:
\begin{equation}
  \label{eq:jump_formulas}
  \ddn {\Sglf D \phi} \Big\vert_\pm  = \Paren{\pm\frac 1 2 I + \Kstar D}[\phi].
\end{equation}
We also introduce the $L^2$-adjoint of $\mathcal{K}^\star_D$, $\mathcal{K}_D$, which is given by
$$
\mathcal{K}_D[\phi] (x) := \int_{\p D} \frac
{\partial \Gamma}{\partial {\nu(y)}}(x-y) \phi (y)\,ds(y), \quad
\phi \in \LtpD.
$$

\subsection{Preliminary results}
\label{sec:preliminary-results}

We recall first that the operator $\Kstar D$ is compact, provided that $D$ is of class $\mathcal{C}^{1,\alpha}$ for some $0<\alpha<1$, with eigenvalues included in
$(-\hlf, \hlf]$ and it can be decomposed as \cite{ammari_mathematical_2013}
\begin{align}
  \label{eq:SVD_Kstar}
  \Kstarf D \phi = \sum_{j=1}^\infty \mu_j \seqS{\phi, u_j} \,u_j,
\end{align}
where $\mu_j$ and $u_j\in\LtpD$ are the $j$-th eigenvalue and eigenvector of $\Kstar D$
respectively, and the scalar product
\begin{align}
  \label{eq:innerprod_SD}
  \seqS{\phi, u_j}:=\intbd D {\phi(y) \Sglf D {u_j}(y)} y .
\end{align}
Furthermore, we have the the energy identity 
\begin{align}
  \label{eq:energyident_SD}
  \norm{\hat\phi}_\LtpD^2 = \sum_j\abs{\seqS{\hat\phi,u_j}}^2.  
\end{align}
The spectral decomposition (\ref{eq:SVD_Kstar}) is based on a Calder\'on's identity and a symmetrization principle; see for instance \cite[Chap. 2]{ammari_mathematical_2013}.  

\begin{lem}\label{lem:loKstar_on_Schwarz}
  Let $\vphi\in\SzRLD$ and let $\hphi$ be its Fourier Transform. The mapping
  \begin{align}
    \label{eq:def_op_loKstar}
    \hphi(\omega) \mapsto \loKstarf D {\hphi(\omega)}, \ \forall \omega\in\R
  \end{align}
  defines a homeomorphism on $\SzRLD$, and in particular,
  \begin{align}
    \label{eq:continuity_loKstar}
    \hphi\xrightarrow{\Sz} 0 \ \text{ implies }\ \loKstarfi D \hphi \xrightarrow{\Sz} 0.
  \end{align}
  The same results hold also for the operator $\loKnst D$.
\end{lem}
\begin{proof}
  We shall prove the lemma only for $\loKstar D$. The case of the operator $\loKnst D$ is
  similar.
  
  


  For $b\in\NN$, let $\lambda^{(b)}(\omega)$ be the derivative of order $b$ of
  $\lambda(\omega)$ in $\omega$, then any $\hphi\in\SzRLD$ multiplied by
  $\lambda^{(b)}(\omega)$ remains a function of $\SzRLD$. Moreover, by applying the product
  rule and the boundedness of $\Kstar D$, it is easy to verify
  \begin{align*}
    \pab\Paren{\loKstarf D \hphi} \lesssim \sum_{0\leq b'\leq b}p_{a,b'}(\hphi)<\infty, \ \forall
    a,b\in\NN
  \end{align*}
  and hence,  $\loKstarf D \hphi\in\SzRLD$ for any $\hphi\in\SzRLD$. 
  
  For a fixed $\omega$, the operator $\loKstar D$ is invertible on $\LtpD$. Hence
  \begin{align*}
    \loKstarf D \hphi = 0, \ \forall \omega\in\R 
  \end{align*}
  implies $\hphi(\omega)=0, \forall \omega$, thus $\hphi=0$ in
  $\SzRLD$. Therefore $\loKstar D$ is injective.

  To prove that $\loKstar D$ is surjective, it suffices to show that $\loKstari D$ maps $\SzRLD$ to
  $\SzRLD$. The following statement can be verified easily. For $k\in\N$, we have
  \begin{align*}
    \Paren{\loKstarfik D k \hphi}' = \loKstarfik D k {\hphi'} - k\lambda'(\omega)\loKstarfik D
    {(k+1)} \hphi , \ \forall \omega\in\R,
  \end{align*}
  and more generally,
  \begin{align*}
    \Paren{\loKstarfi D \hphi}^{(b)} = 
    \sum_{0\leq b'\leq b+1} \Paren{\lambda(\omega) I - \Kstar D}^{-{b'}} \Brack{P_{b'}(\hphi(\omega);
      \lambda(\omega))}
  \end{align*}
  where $P_{b'}$ is a differential operator of order $b+1$ in $\omega$ with coefficients
  depending on $\lambda(\omega)$ and its derivatives (up to order $b+1$). Furthermore, 
  \begin{align*}
    \norm{\loKstari D}_\LtpD \leq \frac 1 {\abs{\lambda(\omega)} - 1/2}
  \end{align*}
  which behaves as $O(\abs{\omega})$ only when $\omega\to\infty$, therefore it holds
  \begin{align*}
    \pab\Paren{\loKstarfi D \hphi} \lesssim \sum_{0\leq a',b'\leq
      a+b+1}p_{a',b'}(\hphi)<\infty, \ \forall a,b\in\NN
  \end{align*}
  Hence $\loKstar D$ is surjective.

  Finally, the claim \eqref{eq:continuity_loKstar} follows from the inequality above
  and this completes the proof.
\end{proof}

The following result shows that the operator $\loKstari D$ can preserve causality, at
least for some special class of functions such as separable functions: $\psi(t,x) =
h(t)\tpsi(x)$ for some function $h$ of the classical Schwartz space $\Sz(\R)$ and $\tpsi$ of
$\LtpDo$.  Here, $\LtpDo$ is the set of functions in $L^2(\partial D)$ with zero mean-value.


\begin{thm}
  \label{thm:separable_causality}
  For a separable and causal function $\psi\in\SzRLDo$, define a function $\vphi$ in the
  frequency domain as
  \begin{align}
    \label{eq:loKstar_causality}
    \hphi = \loKstarfi D \hpsi.
  \end{align}
  Then $\vphi\in\SzRLDo$ and $\vphi$ is causal.
\end{thm}
\begin{proof}
  The fact that $\vphi\in\SzRLDo$ follows from Lemma \ref{lem:loKstar_on_Schwarz} and the
  property that $\loKstar D$ is a bijection on $\LtpDo$.  

  For fixed $\omega$, the singular value decomposition gives:
  \begin{align}
    \label{eq:svd_loKstari}
    \hphi(\omega) = \loKstarfi D {\hpsi(\omega)} = \sum_{j} \frac {\seqS{\hpsi,
        u_j}}{\lambda(\omega) - \mu_j} u_j
  \end{align}
  where $\abs{\mu_j}<\hlf$ and $u_j\in\LtpDo$ are the $j$-th eigenvalue and eigenvector of
  $\Kstar D$ respectively and are independent of $\omega$. Notice that
  \begin{align*}
    \frac 1 {\lambda(\omega)-\mu_j} = \alpha_j\Paren{1 - \frac{\beta_j}{\gamma_j + i\omega}},
  \end{align*}
  with the constants $\alpha_j=\frac{2}{1-2\mu_j}, \beta_j={\alpha_j}/{\e}$ and
  $\gamma_j=\sigma/\e + \frac{1+2\mu_j}{\e(1-2\mu_j)}>0$. Let
  \begin{align}
    \label{eq:time_conv_kernel}
    g_j(t) = \mathbbm{1}_{t\geq 0}(t) e^{-\gamma_j t},
  \end{align}
  whose Fourier transform is $\hat g_j(\omega) = 1/(\gamma_j+i\omega)$. Then the function
  $(\lambda(\omega)-\mu_j)^{-1}\hpsi(\omega)$ in the time domain is
  \begin{align*}
    \alpha_j \psi(t) - \alpha_j\beta_j g_j*\psi(t),
  \end{align*}
  which is clearly a causal function. Hence it suffices to show that the sum
  in \eqref{eq:svd_loKstari} converges in $\SzRLD$. Then by taking inverse Fourier transform
  term by term we obtain the causality of $\vphi$. 
  %
  For doing so, we write for given
  $a,b\in\NN$
  \begin{align*}
    \Paren{\pab\Paren{\sum_{j=N}^\infty \frac {\seqS{\hpsi, u_j}}{\lambda(\omega) - \mu_j}
        u_j}}^2 &= \sup_{\omega\in\R}\, \abs{\omega}^{2a}\Norm{\sum_{j=N}^\infty \Paren{\frac
        {\seqS{\hpsi,
            u_j}}{\lambda(\omega) - \mu_j}}^{(b)} u_j}_\LtpD^2\\
    &= \sup_{\omega\in\R} \,\abs{\omega}^{2a}\sum_{j=N}^\infty \Abs{\Paren{\frac {\seqS{\hpsi,
            u_j}}{\lambda(\omega) - \mu_j}}^{(b)}}^2,
  \end{align*}
  where the $b$-th order derivative in the first identity is taken termwise since the
  derivative is a continuous linear mapping on $\SzRLD$.
  It is easy to see that it will be bounded for any $a,b\in\NN$ if 
  \begin{align*}
    \sup_{\omega\in\R}\, \abs{\omega}^{2a} \sum_{j=N}^\infty \Abs{\seqS{\hpsi^{(b)}, u_j}}^2
    < \infty, \ \forall a, b \in \NN,
  \end{align*}
  which is indeed the case since $\psi(t,x)=h(t)\tpsi(x)$ with $h\in\SzR$ and
  $\tpsi\in\LtpDo$. Moreover due to the energy identity \eqref{eq:energyident_SD}, the last
  expression tends to 0 as $N\to\infty$. This proves the convergence of \eqref{eq:svd_loKstari}
  in $\SzRLD$. The proof of the theorem is then complete.

\end{proof}


\subsection{Integral representation and an existence result}
\label{sec:repr-solution}

We denote in the following
\begin{align}
  \label{eq:background_field}
  U(t,x) = h(t)\tU(x) = h(t)\sum_{j=1}^p a_j\Gamma(x-\xsj),
\end{align}
which is a solution to $\Delta U(t,x) = h(t)\tf(x) = f(t,x)$ and decays as $O(\abs{x}^{1-d})$
when $\abs x$ goes to infinity, due to condition \eqref{eq:neutrality}.

\begin{thm}\label{thm:repr_solution}
  Let $\alpha,\lambda, \lambda(\omega)$ be defined as in \eqref{eq:some_constants}. For the
  source term \eqref{eq:source_f} with $h$ fulfilling \eqref{eq:h_assump1}, the unique solution
  to \eqref{eq:governing_eq_trans} is given by
  \begin{align}
    \label{eq:repr_solution}
    u(t) = U(t) + \Sglf{D}{\vphi(t)},
  \end{align}
  where $\vphi\in\SzRLDo$ is causal and solves the following equation:
  \begin{align}
    \label{eq:repr_vphi_func_fulfill}
    \lKstarf D {\vphi} + \alpha \hKstarf D{\vphi'} = (1+\alpha \p_t) \frac{\partial U}{\partial \nu},
  \end{align}
  or equivalently in the frequency domain
  \begin{align}
    \label{eq:repr_hat_vphi_fulfill}
    \Paren{\lambda(\omega)I - \Kstar D}[\hphi] = \frac{\partial \hU}{\partial \nu}.
  \end{align}
  Furthermore, the solution \eqref{eq:repr_solution} is causal and belongs to $\HoRpX$.
\end{thm}
\begin{proof}  
  %
  %
  %

  For $u$ given by \eqref{eq:repr_solution}, one can check easily that the first and second
  identies in \eqref{eq:governing_eq_trans} are verified. Further, since $U$ and $\Sglf D
  {\vphi(t)}$ are both continuous across the boundary, the third identity also holds true. The
  fourth identity in \eqref{eq:governing_eq_trans} is equivalent to
  \begin{align*}
    \vep\Paren{\frac{\partial}{\partial \nu} \Sglf D \vphi}'\Vm + \sigma \Paren{\frac{\partial}{\partial \nu}\Sglf D \vphi}\Vm
    - \Paren{\frac{\partial}{\partial \nu}\Sglf D \vphi}\Vp =
    (1-\sigma)\frac{\partial U}{\partial \nu} - \vep \frac{\partial U'}{\partial \nu},
  \end{align*}
  which becomes \eqref{eq:repr_vphi_func_fulfill} by applying the jump formula
  \eqref{eq:jump_formulas} and by interchanging the derivative and the single layer
  potential. Taking Fourier transform in the $t$-variable  in \eqref{eq:repr_vphi_func_fulfill}
  yields \eqref{eq:repr_hat_vphi_fulfill} after some simplifications.

  In the time domain, the term on the right-hand side of \eqref{eq:repr_hat_vphi_fulfill} corresponds to
  $h(t) \frac{\partial \tU(x)}{\partial \nu} $ which is separable, causal, and belongs to $\SzRLDo$. Therefore by
  Corollary \ref{thm:separable_causality} the function $\vphi\in\SzRLDo$ is causal. This proves
  the causality of the solution $u$, as well as the fifth identity in
  \eqref{eq:governing_eq_trans}, since $\Sglf D {\vphi(t)}$ decays as $O(\abs{x}^{1-d})$ for
  $\vphi(t)$ being an $L^2_0(\p D)$ function.

  Finally, $\vphi\in\SzRLDo$ being causal implies $\vphi(0)=0$ so the last identity in
  \eqref{eq:governing_eq_trans} is also fulfilled. 

  It is clear that $U\in\HoRpX$. To prove $u\in \HoRpX$, it suffices
  to show for any compact $K\subset \R^d$ the boundness of:
  \begin{equation}
    \label{eq:SglHoBoundness}
    I_1+I_2 = \int_{\R_+}\Norm{\Sglf D {\vphi(t)}}_{H^1(K)}^2\, dt + \int_{\R_+}\Norm{\Sglf D
      {\vphi'(t)}}_{H^1(K)}^2\, dt.
  \end{equation}
  Note that
  \begin{align*}
    I_1 = \int_{\R_+} \int_K\Abs{\Sglf D {\vphi(t)} (x)}^2\, dx + \int_{\R_+}\int_K\Abs{\nabla\Sglf D
      {\vphi(t)} (x)}^2\, dx,
  \end{align*}
  and  the first term in $I_1$ can be estimated as
  \begin{align*}
    \int_{\R_+} \int_K \Abs{\intbd{D}{\Gamma(x-y)\vphi(t,y)}{y}}^2\, dx\, dt \leq
    \int_{\R_+} \int_K \Norm{\Gamma(x-\cdot)}_{L^2(\p D)}^2 \Norm{\vphi(t)}_{L^2(\p D)}^2
    \, dx \, dt
  \end{align*}
  and is bounded since the singularity of $\Gamma$ is integrable and $\vphi$
  is a function of $\SzRLDo$. Similarly one can prove the boundedness for the other terms,
  therefore $u\in \HoRpX$.

  The uniqueness of the expression is a consequence of Lemma
  \ref{lem:uniqueness_of_gov_eq} and then the well-posedness of
  \eqref{eq:governing_eq_trans} is now established. \end{proof}



\section{Time-dependent GPTs and asymptotic expansions}
\label{sec:time-depend-polar}

In this section we extend the concept of generalized polarization tensor (GPT) to the time
domain\footnote{The GPT as it is defined in this paper is actually the so-called
  \emph{contracted} GPT introduced in \cite{ammari_enhancement_2013}.}. The GPTs will be
the features of the target to be recovered from measurements. For the sake of simplicity,
we only discuss the two-dimensional case here. The three-dimensional case can be treated
by following the same approach as in \cite{ammari_invariance_2014}.

For the domain $D$ and the order $m,n\in\N$, the GPT in the frequency domain (at the frequency
$\omega$) is a $2\times 2$ matrix of the following form \cite{ammari_target_2014}
\begin{align}
  \label{eq:def_Mmn}
  \wMmn = \wMmn(\omega; D) =
  \begin{pmatrix}
    \wMcc & \wMcs\\
    \wMsc & \wMss
  \end{pmatrix},
\end{align}
where $\wMcs$ is defined as
\begin{align}
  \label{eq:def_Mcc}
  \wMcs(\omega;D) = \intbd D {S_n(y) \loKstarfi D {\frac{\partial C_m}{\partial \nu}}(y)} y
\end{align}
with $C_m$ and $S_m$ being respectively the real and imaginary parts of the harmonic
polynomial $(x_1+ix_2)^m$, and $\lambda(\omega)$ being defined as in
\eqref{eq:some_constants}. The other terms $\wMcc, \wMsc, \wMss$ in \eqref{eq:def_Mmn} are
defined in a similar way, by replacing the symbols $c$ and $s$ by the corresponding
polynomials $C_m$ (or $C_n$) and $S_m$ (or $S_n$) respectively. The time-dependent GPTs
$\Mmn(t;D)$ is also a $2\times 2$ matrix consisting of the inverse Fourier transform (in the
sense of distribution) of each term of $\wMmn(\omega;D)$.

In the following we denote by $\wM=\wM(\omega;D) = (\wMmn)_{mn}$ the block matrix of the GPTs
in the frequency domain, and $\mM = \mM(t; D) = (\Mmn)_{mn}$ in the time domain.

\subsection{Properties of the time-dependant  GPTs}
\label{sec:prop-time-depend}

The operator $\loKstari D$ is uniformly continuous in $\omega$ with respect to the operator
norm $\norm{\cdot}_\LtpD$, and converges to $(\hlf I - \Kstar D)^{-1}$ as $\omega$ tends to
infinity. In the limit case, $\wMmn$ becomes independent of the frequency but remains well
defined since $(\hlf I - \Kstar D)$ is invertible on $L^2_0(\p D)$. Hence we obtain the
following result.

\begin{prop}\label{prop:wM_boundness}
  For any $m,n\in\N$ and as a function of $\omega$, each entry of $\wMmn(\omega;D)$ is
  uniformly continuous and bounded. Furthermore,
  \begin{align*}
    \lim_{\abs{\omega} \rightarrow \infty}\wMmn(\omega; D) = \wMmn(\infty;D),
  \end{align*}
  where $\wMmn(\infty;D)$ is some well-defined matrix.
\end{prop}

\paragraph{GPT as distribution}
\label{sec:contracted-gpts-as}
For a general shape $D$ its GPT $\wMmn(\omega;D)$ does not exhibits any decay as $\omega$
tends to infinity, and we interpret the time domain $\Mmn$ as a distribution in
$\SzRp$. Furthermore, the entries of $\wMmn$ are $\Loloc$ functions, so we define the
action of $\mMcs$ in the frequency domain as
\begin{align}
  \label{eq:action_wMcs_distrib}
  \dualSSp{\wMcs, \vphi} := \int_\R \vphi(\omega) \wMcs(\omega) \, d\omega
\end{align}
and similarly for the other entries $\mMcc, \mMsc,$ and  $\mMss$.

\begin{prop}
  \label{prop:CGPT_as_causal_distrib}
  The distribution $\mMmn(t;D)\in\SzRp$ is causal, which means that for any causal function
  $\vphi\in\SzR$, 
  \begin{align}
    \label{eq:def_causal_distrib}
    \dualSSp{\mMmn, \tphi} = 0, \ \text{ where } \tphi(t) := \vphi(-t),
  \end{align}
  holds.
\end{prop}
\begin{proof}
  We prove the result for $\mMcs$ only. The result for the other entries can be proved similarly. 
  By Fourier transform of the distribution we have 
  $$ 2\pi \dualSSp{\mMcs, \tphi} =
  \dualSSp{\wMcs, \hphi},$$ and by \eqref{eq:action_wMcs_distrib},
  \begin{align*}
    \dualSSp{\wMcs, \hphi} &= \intbd D {S_n(y)
      \int_\R \underbrace{\loKstarfi D {\hphi(\omega)\frac{\partial C_m}{\partial \nu}} (y)}_{\hpsi} \,d\omega} y,
  \end{align*}
  where the function $\hphi \frac{\partial C_m}{\partial \nu}\in\SzRLDo$ in the time domain is separable and
  causal. By Corollary \ref{thm:separable_causality}, the function $\psi\in\SzRLDo$ defined in
  the expression above via $\hpsi$ is causal in the time domain, hence
  \begin{align*}
    \dualSSp{\mMcs, \tphi} &= \seqD{S_n, \psi(0)} = 0,
  \end{align*}
  due to the fact $\psi(0)=0$. This completes the proof.
\end{proof}

\subsection{Asymptotic expansion}
\label{sec:asympt-expans}


Taking the Fourier transform of the representation formula \eqref{eq:repr_solution}, it follows that
\begin{align*}
  \hat u(\omega, x) = \hU(\omega, x) + \Sglf D {\hat\vphi(\omega)}(x),
\end{align*}
and since $\lambdao I -\Kstar D$ is invertible, plugging \eqref{eq:repr_hat_vphi_fulfill} into
the identity above yields 
\begin{align}
  \label{eq:hat_repr_solution}
  \hat u(\omega, x) = \hU(\omega, x) + \hat h(\omega) \Sglf D{\loKstarfi D {\frac{\partial \tU}{\partial \nu}}}(x).
\end{align}

Let $z\in\R^d$ be an estimated position of the target $D$. 
For the source $x_s=\set{x_s^1,\ldots,x_s^p}$ with $\xsj\in \Dc$ and the receiver
$x_r\in\Dc$, let $(\rhosj,\thesj)$ and $(\rho_r, \theta_r)$ be the polar coordinate of
$\xsj-z,j=1\ldots p$ and $x_r-z$ respectively. We introduce the $1\times 2$ matrices
\begin{align}
  \label{eq:Asm_Brn}
  A_{sm} = \sum_{j=1}^p \frac{a_j}{2\pi m \rhosj}
  \begin{pmatrix} 
    \cos(m\thesj) & \sin(m\thesj)
  \end{pmatrix},\ \ 
  B_{rn} =     
  \frac{1}{2\pi n \rho_r}
  \begin{pmatrix}
    \cos(n\theta_r) & \sin(n\theta_r).
  \end{pmatrix}
\end{align}
Then by expanding the fundamental solution $\Gamma$  in \eqref{eq:hat_repr_solution} into its Taylor series as
done in \cite{ammari_target_2014}, we can establish an asymptotic expansion relating the data
with the GPTs:
\begin{align}
  \label{eq:GPT_expansion}
  \hu(\omega,x_r) - \hU(\omega, x_r) = \sum_{m,n=1}^K A_{sm} \hh(\omega) \wMmn(\omega;
  D-z)B_{rn}^\top + E_K,
\end{align}
where $D-z$ denotes the translation of $D$ by the vector $-z$, $K$ is the truncation order and
$E_K$ is the truncation error which decays exponentially to $0$ as $K$ increases
\cite{ammari_target_2014}. 

\subsection{Linear system}
\label{sec:linear-system}

In the time domain, the perturbation of the field corresponding to the source $x_s$ and recorded
by the receiver $x_r$ constitutes the $(s,r)$-th entry of the multi-static response  (MSR) matrix   $\mV(t) =
(\Vsr(t))_{sr}$ at the time $t$:
\begin{align}
  \label{eq:Vsr_def}
  \Vsr(t)=u(t,x_r) - U(t,x_r),
\end{align}
and its Fourier transform in $t$ is just the term on the left-hand side of \eqref{eq:GPT_expansion} that we
denote by $\mhV(\omega) =(\hVsr(\omega))_{sr}$. By introducing a linear operator $\mL$ in
\eqref{eq:GPT_expansion} and dropping the truncation error $E_K$, we can rewrite it as a
linear system:
\begin{align}
  \label{eq:Vsr_linsys_freq}
  \mhV(\omega) \simeq \mL(\hh(\omega)\wM(\omega;D-z)),
\end{align}
where $\wM(\omega;D-z)$ is a $2K \times 2K$ block matrix. Remark that the operator $\mL$
depends only on the measurement system (\ie the reference point $z$, the sources $x_s$ and
receivers $x_r$) and the truncation order, and that the data $\mhV(\omega)$ or $\mV(t)$
can be contaminated by some white noise.

\subsubsection{Filtered GPT}
\label{sec:filtered-gpt}

By Proposition \ref{prop:CGPT_as_causal_distrib} the GPT $\mM$ in the time domain is a
distribution, however $\mM$ ``filtered'' by $h$ becomes a
regular function. To show this we introduce the concept of \emph{Filtered GPT}:
\begin{dfn}
  \label{dfn:FGPT}
  The filtered GPT $\mNmn(t;D)$ in the time domain is a $2\times 2$ matrix which corresponds in
  the frequency domain to
  \begin{align}
    \label{eq:wMmnh_freq}
    \wNmn(\omega;D) =  
    \begin{pmatrix}
      \wNcc & \wNcs\\
      \wNsc & \wNss
    \end{pmatrix}= \hh(\omega) \wMmn(\omega;D).
  \end{align}
\end{dfn}

\begin{prop}
  \label{prop:Nmn_hM}
  Let $h\in\SzR$ be causal and the filtered GPT $\mNmn(t;D)$ defined as in Definition
  \ref{dfn:FGPT}.  Then each entry of $\mNmn$ in the time domain is causal and belongs to
  $\SzR$.
\end{prop}
\begin{proof}
  We prove this result only for the entry $\wNcs$. By definition
  \begin{align*}
    \wNcs(\omega) := \hh(\omega)\wMcs(\omega;D) = \intbd D {S_n(y) \underbrace{\loKstarfi D
        {\hh(\omega) \frac{\partial C_m}{\partial \nu}}}_{\hphi(\omega)}(y)} y,
  \end{align*}
  where $\hphi$ in time domain is causal and a function of $\SzRLDo$, as a consequence of
  Corollary \ref{thm:separable_causality}. It is easy to check that the inner product
  $\seq{S_n, \cdot}_{L^2(\p D)}$ defines a continuous linear mapping from $\SzRLD$ to $\SzR$,
  hence $\wNcs$ as well as $\Ncs$ is in $\Sz(\R)$.
\end{proof}

In the following we denote the block matrix $\mN=\mN(t;D)=(\mNmn)_{mn}$, then the linear
system \eqref{eq:Vsr_linsys_freq} can be rewritten in the time domain as
\begin{align}
  \label{eq:Vsr_linsys_time}
  \mV(t) \simeq \mL(\mN(t; D-z)). 
\end{align}
Although the two linear systems \eqref{eq:Vsr_linsys_freq} and \eqref{eq:Vsr_linsys_time}
are equivalent, in practice it is prefered to consider \eqref{eq:Vsr_linsys_time} since
the measurements are taken directly in the time domain. By inverting $\mL$ one can
estimate $\mN(t;D-z)$ from data, and the results in \cite{ammari_target_2014} about the
maximum resolving order as well as the stability remain valid here.

\begin{rmk}
  Notice that one cannot expect to recover stably the GPT $\mM$ from the filtered GPT
  $\mN$ by a deconvolution procedure, since the pulse-type signal $h$ in practice is always
  band-limited, while in general $\mM$ is not band-limited function, as shown in
  Proposition \ref{prop:wM_boundness}. 
\end{rmk}


\section{Shape identification with pulse-type signals}
\label{sec:shape-ident-with}

We aim to identify a target $D$ from a dictionary of reference shapes
$\set{B_1\ldots B_N}$ up to some rigid transformation and dilation. In this section we propose
a time domain multi-scale method for shape identification. For the sake of simplicity, we
assume that the target and all reference shapes have the same physical parameters
$\sigma, \e$, 
which can be estimated from data via a nonlinear parameter fitting procedure
as described in \cite{ammari_modeling_2013}.

\subsection{Invariant properties of the filtered GPTs}
\label{sec:invar-prop-conv}

In \cite{ammari_target_2014} and \cite{ammari_shape_2014} the properties of the GPTs
$\hat\mM(\omega;D)$ with respect to the scaling and rigid motion have been
investigated. The filtered GPTs $\mN(t;D)$ being defined in the frequency domain as
$\hat\mN(\omega;D)=\hat h(\omega) \hat\mM(\omega;D)$ inherit naturally all of these
properties. The following result is a direct consequence of the results in  \cite{ammari_target_2014} and \cite{ammari_shape_2014} and its proof is skipped here.

\begin{prop}
  \label{prop:invar-prop-time}
  The matrix of the filtered GPTs $\mN(t; D)$ is symmetric. Moreover,
  for arbitrary $z\in\R^d$, $s>0$ and $R\in SO(\R^d)$, with $SO(\R^d)$ being the rotation group in $\R^d$, the following identity holds for  the $d \times d$ square matrix $\mN_{11}$:
  \begin{align}
    \label{eq:invariant_N}
    \mN_{11}(t; z+sRD) = s^d R\, \mN_{11}(t; D)\, R^\top.
  \end{align}  
  Furthermore, the singular values of  $\mN_{11}(t; D)$ fulfills
  \begin{align}
    \label{eq:invariant_fro_N}
    \tau_n(t; z+sRD) = s^d \tau_n(t; D),\ \ n=1\ldots d . 
  \end{align}
\end{prop}
We assume for the rest of the paper that the singular values are sorted in a decreasing order:
$\tau_1(t; D)\geq \tau_2(t;D)\ldots\geq \tau_d(t;D)\geq 0$.

\subsection{Shape descriptors based on the polarization tensor}
\label{sec:features-based-first}


In \cite{ammari_target_2014} the authors constructed the GPT-based shape descriptors
applicable for the shape identification in electro-sensing. These descriptors have
infinite orders and allow to distinguish between complex shapes using only one frequency.
Nonetheless, this approach requires high order GPTs (\eg,
$\wMmn(\omega;D)$ for $m,n\geq 2$) which are difficult to obtain in practice, for example with
far field and limited angle of measurement view. It has then limited feasibility. 

The situation here for the filtered GPTs $\mN$ is identical. In fact, the total error of
reconstruction at the order $K$ is the sum of the error due to the truncation
$O(\rho^{-(K+2)}$) and the error due to the noise $O(\rho^{K}/N_s)$, with $N_s$ being the
number of equally distributed transmitters and $\rho>1$ the ratio between the
transmitter-to-target distance and the size of the target. So the reconstruction of high
order information is exponentially unstable, which is contrasted with the fact that at low
orders the error due to the noise can be reduced to zero by increasing the number of
transmitters \footnote{This is in agreement with the biological evidence that the weakly
  electric fish's skin is densely covered by the electrical receptors.}. Numerical
experiments in \cite{ammari_shape_2014} confirmed that with a large number of transmitters
the reconstruction of the polarization tensor (or the first order GPT) is very stable for
various settings of measurement system. On the other hand, it is known that the
reconstruction of GPTs of order greater than one is extremely unstable when the angle of
view is limited \cite{ammari_tracking_2013}.

The fundamental limit of using the polarization tensors in shape description is that they
do not contain high order information of the shape and can only describe (at a fixed
frequency) an equivalent ellipse \cite{ammari_polarization_2007}. However when probed with
a range of frequency, distinct shapes have different response which is the basis of the
multi-frequency approach proposed in \cite{ammari_shape_2014}. We propose here a
multi-scale construction of shape descriptors in the time domain that exploits the first
order filtered polarization tensor $\mN_{11}$ at different frequency band by varying the
pulse shape $h$. The new shape descriptors can describe complex shapes and contain both
the temporal and frequency signature of a shape. Furthermore, they are particularly
robust as we will see in Section \ref{sec:numer-exper} by numerical experiments.

\subsubsection{Multi-Scale invariants}
\label{sec:mscale-invariants}

Assume that $h$ is a band pass filter such that $\hh(0)=0$ (such function can be easily
obtained from derivatives of a Gaussian, for example), and 
let $h_j$ be the dyadic dilation of $h$ at the scale $j$:
\begin{align}
  \label{eq:def_hj}
  h_j(t) = 2^{j/2} h(2^j t) \ \text{ and }\ \hh_j(\omega) = 2^{-j/2} \hh(2^{-j}\omega).
\end{align}
We choose the normalization here so that the $L^2$ energy of the pulse remains
constant. Figure \ref{fig:Fh} shows an example of pulse shapes $h$ (smooth truncation of the third
derivative of a gaussian) and some scales in the frequency domain.

\def\figwidth{7cm}
\begin{figure}[htp]
  \centering 
  \subfigure{\includegraphics[width=\figwidth]{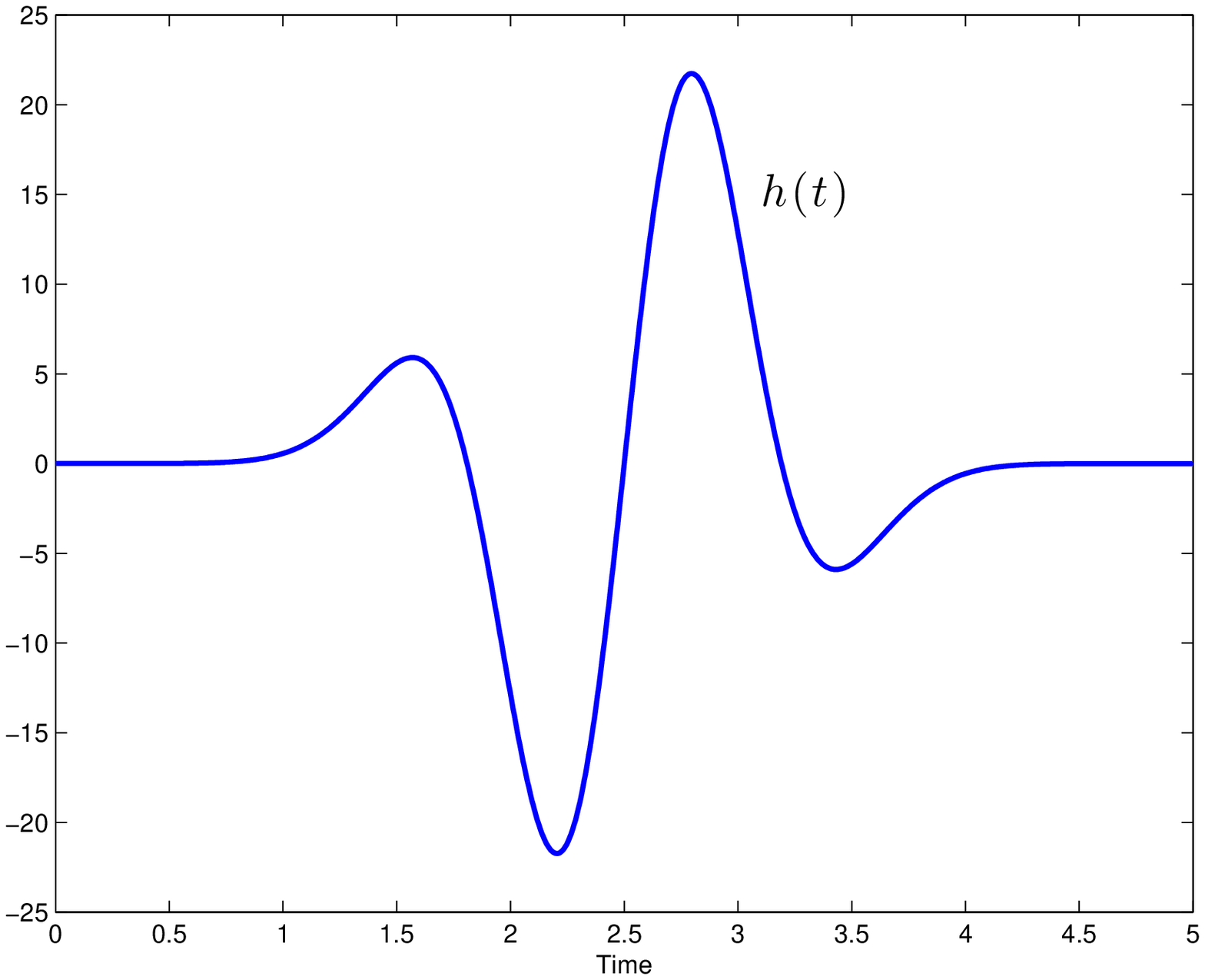}}
  \subfigure{\includegraphics[width=\figwidth]{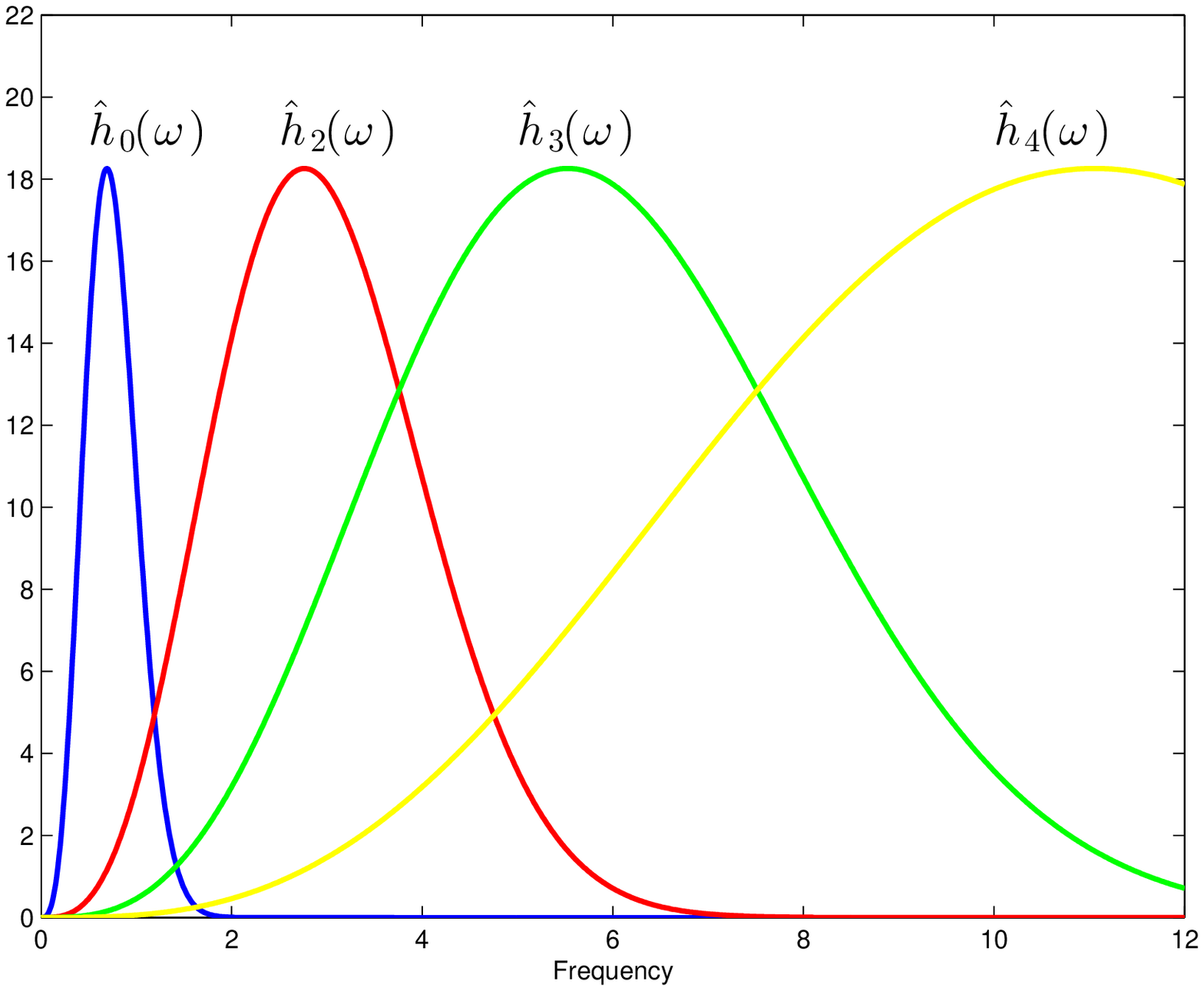}}
  \caption{Example of pulse shape $h$ and some Fourier transforms of $h_j$ (rescaled by $2^{j/2}$).}
  \label{fig:Fh}
\end{figure}

We use $h_j$ as the shape form in the source \eqref{eq:source_f} and acquire for this scale
the filtered GPTs $\mN^j_{11}(t)$, which is the inverse Fourier transform of
$\hh_j(\omega)\wM_{11}(\omega;D)$. Fix $T>0$ the duration of signal acquisition at the
scale $j=0$ and define the quantity
\begin{align}
  \label{eq:def_Ij}
  I_j(t) = I_j(t;D) = \tau_{1}^j(t;D) \Paren{\frac {2^j} T \int_0^{T} \norm{\mN^0_{11}(t;D)}^2_F\,
    dt}^{-1/2},
\end{align}
where $\tau^j_1(t;D)$ is the largest singular value of the matrix $\mN^j_{11}(t;D)$ and
$\norm{\cdot}_F$ denotes the Frobenious norm of a matrix. We remark that the definition
\eqref{eq:def_Ij} is always meaningful since $\mN^0_{11}(t)$ is a smooth function of $t$
and is not identically zero.

It can be seen easily from Proposition \ref{prop:invar-prop-time} that $I_j$ is invariant, in
the sense that for arbitrary $z\in\R^d, s>0, R\in SO(\R^d)$,
\begin{align}
  \label{eq:Ij_invariant}
  I_j(t;z+sRD) = I_j(t;D), \ \forall t>0.
\end{align}

\subsubsection{Shape descriptor}
\label{sec:shape-descriptor}
In order to be processed numerically, $I_j(t)$ is sampled with the step
$\trt_j=2^{-j}T/N$ yielding $N$ equally distributed samples.  We set
\begin{align}
  \label{eq:def_Ijn}
  I_{j,n}(D) = I_j(n\trt_j;D) \simeq \frac{\tau^j_{1}(n\trt_j;D)}{\Paren{\frac{2^j}{N}
      \sum_{n=0}^{N-1} \norm{\mN^0_{11}(n\trt_0;D)}_F}^{1/2}},
\end{align}
and use the concatenation $\mI(D) := \set{I_{j,n}(D)}_{j,n}$ as the shape descriptor of
$D$.  In practice, the number of samples $N$ can be choosen so that the Shannon-Nyquist
sampling condition is fulfilled for the (essential) bandwidth of $h$.

\begin{figure}[htp]
  \centering 
  \subfigure[]{\includegraphics[width=\figwidth]{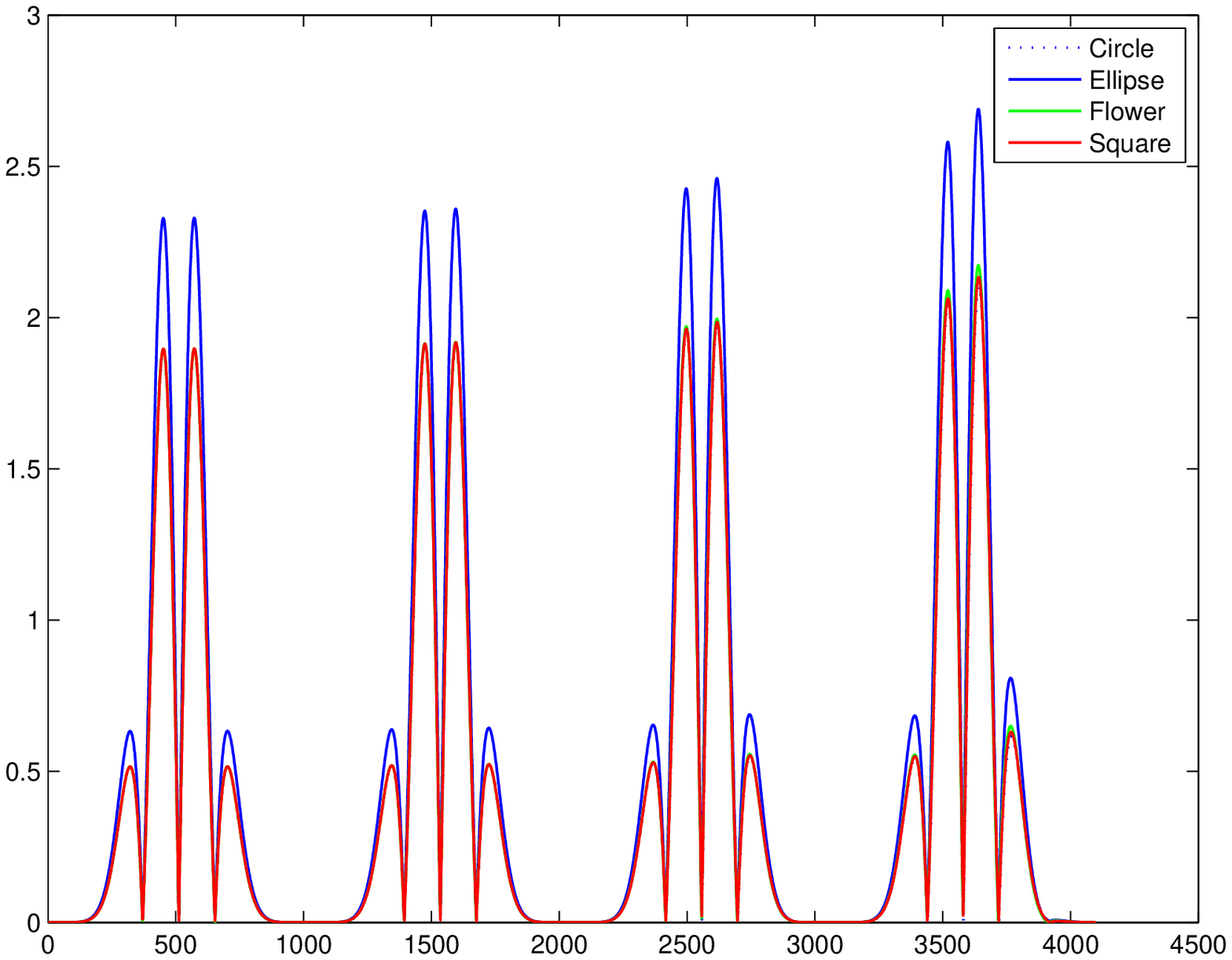}}
  \subfigure[]{\includegraphics[width=\figwidth]{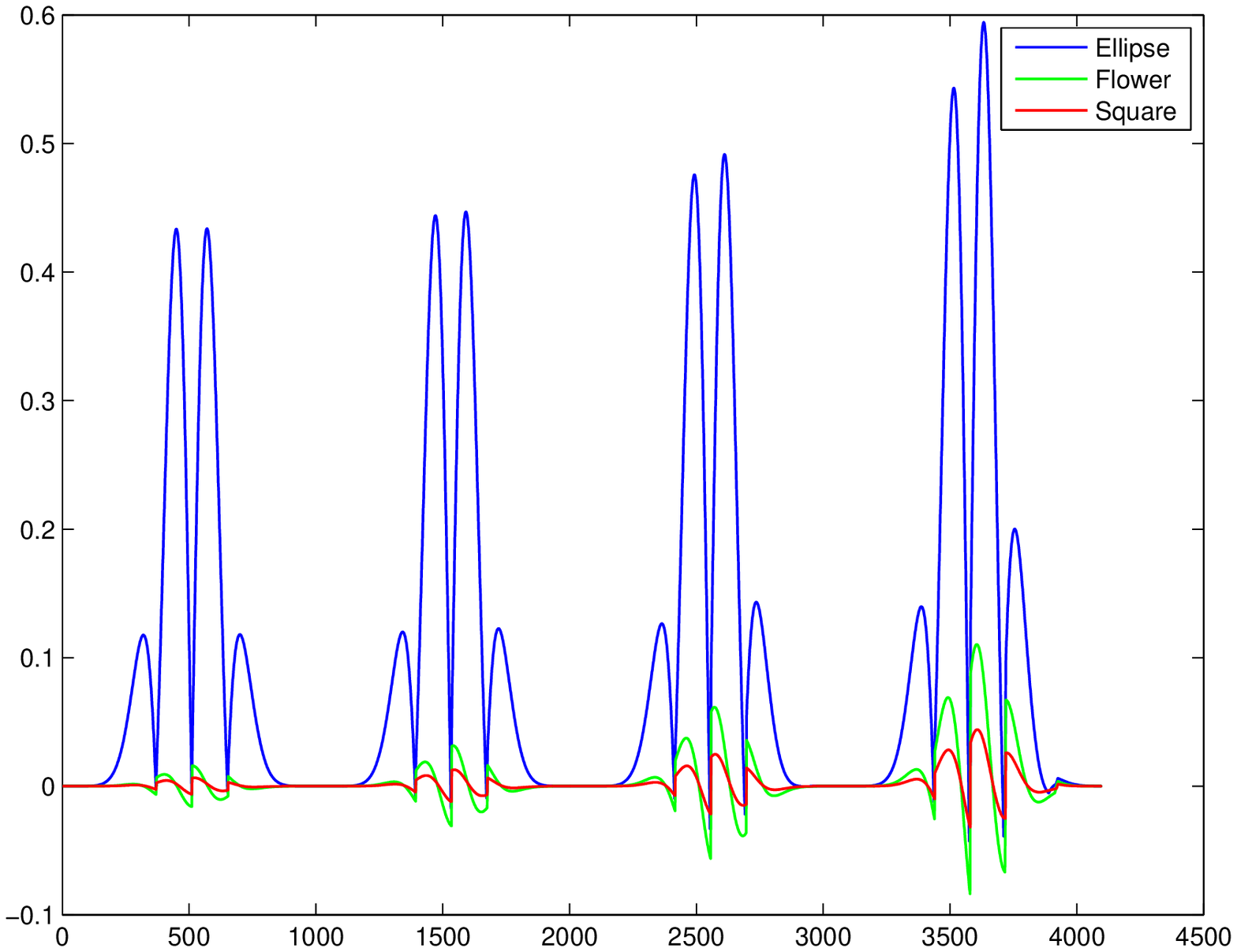}}
  \caption{(a): Shape descriptors $\mI(D)$ of 4 shapes: Circle, Ellipse, Flower and
    Square. (b): Difference of the shape descriptor between Circle and the other shapes.}
  \label{fig:Features}
\end{figure}

Figure.~\ref{fig:Features} shows the shape descriptors corresponding to four shapes,
computed with the pulse shape in Figure \ref{fig:Fh}.(a) at four consecutive scales
$j=-1,0,1,2$. Certain pulse shapes seem to be close to each other and one may ask whether they
allow to distinguish shapes with measurement noise. It turns out, as we shall see in the
next section, that the reconstruction of the filtered polarization tensors is well posed
and the multi-scale shape descriptors obtained from data are robust even at high noise
level. The range of scales $j\in\set{\jmin\ldots\jmax}$ which allows a good distinction between
shapes depends on the dictionary and also on the values of $\sigma,\vep$, and it can be
determined in practice by a numerical optimization procedure.


\section{Numerical Experiments}
\label{sec:numer-exper}

We present in this section some numerical results to illustrate the performance of shape
identification using pulse-type signals. The pulse shape in Figure \ref{fig:Fh} is used as
$h$. The acquisition system consists of $N_s=N_r=50$ positions of transmitters which are
distributed on a circle of radius $10.7$ and centered at the origin. Each source $x_s$ is
composed of two Dirac functions close to each other (within a distance of $0.1$)
satisfying the condition of neutrality \eqref{eq:neutrality}. We will consider only the
limited view case, \ie the transmitters cover uniformly the angle range $[0, \alpha]$ with
$\alpha<2\pi$, as illustrated by Figure~\ref{fig:acqsys}. Such a scenario is close to the
real world situation (the size of the electric fish's body is comparable to that of the
target) and is much harder to solve than the full view case, due to its severe
ill-posedness \cite{ammari_tracking_2013}.


\def\figwidth{7cm}
\begin{figure}[htp]
  \centering 
  \subfigure[$\alpha=\pi/8$]{\includegraphics[width=\figwidth]{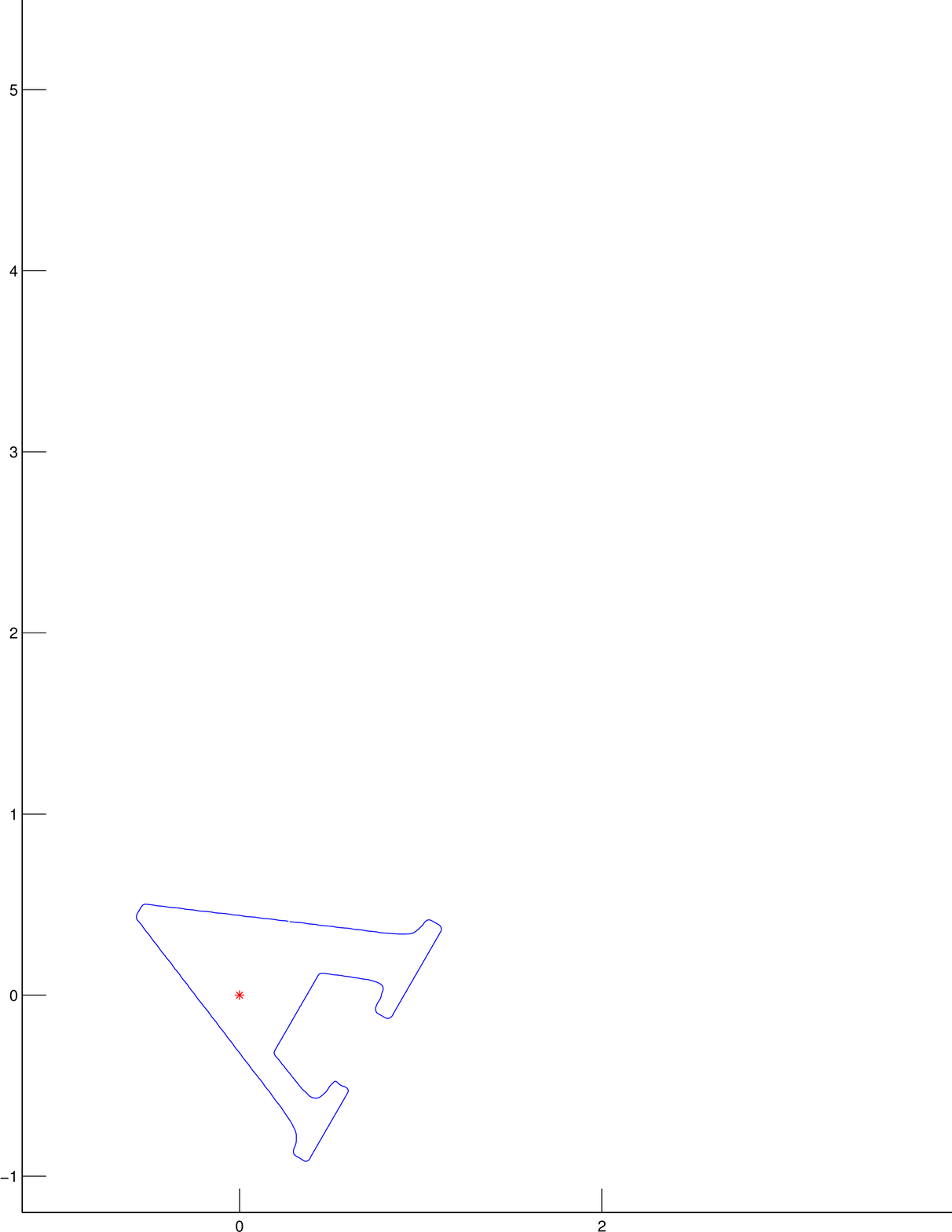}}
  \subfigure[$\alpha=\pi/32$]{\includegraphics[width=\figwidth]{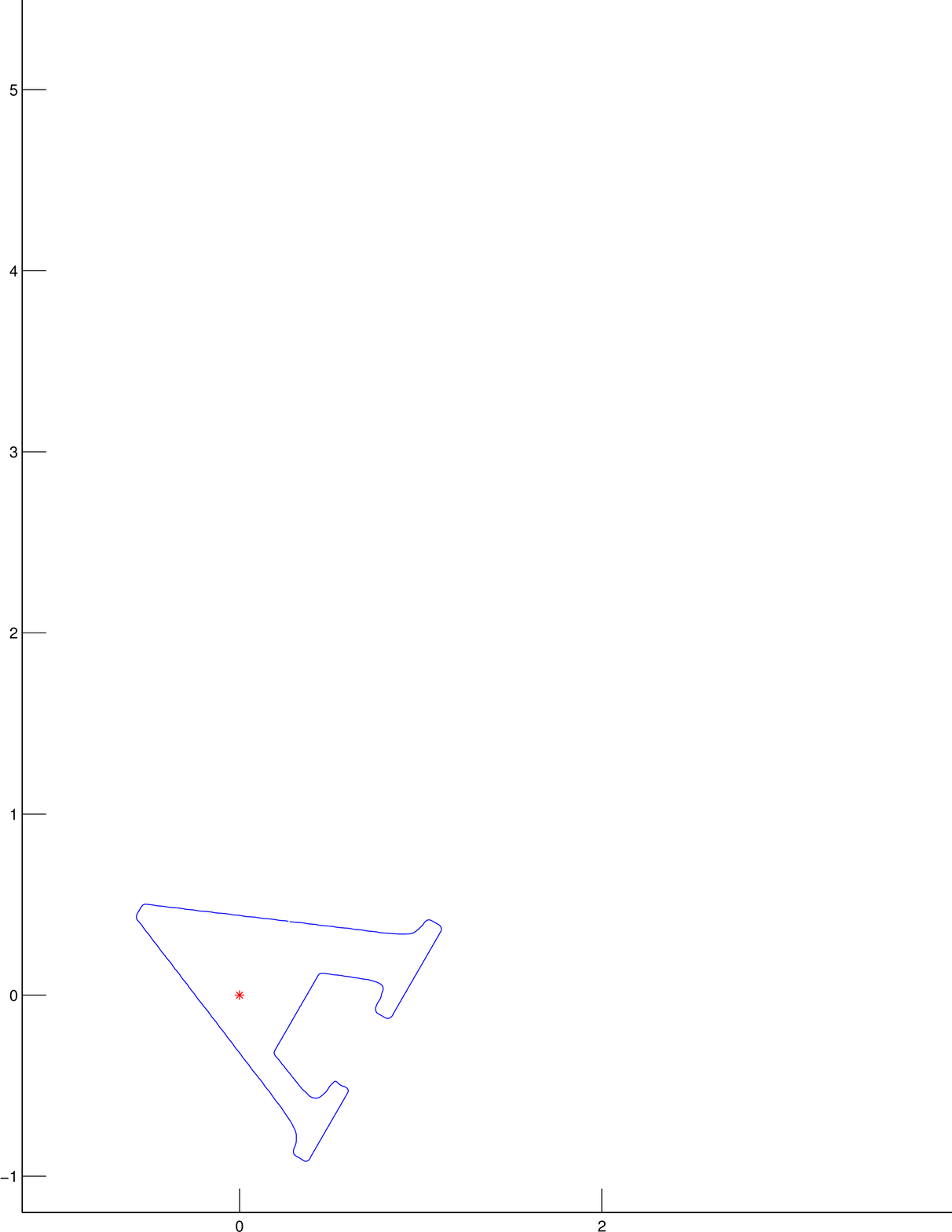}}
  \caption{Examples of acquisition system of limited angle of view using $50$ transmitters
    equally distributed on an arc.  The center of the target is marked by the red '*'.}
  \label{fig:acqsys}
\end{figure}

The overall procedure of the numerical simulation is resumed as follows.

\paragraph{Dictionary}
\label{sec:dictionary}
Our dictionary of standard shapes consists of eight elements $\set{B_n}_{n=1\dots 8}$ as shown
in Figure~\ref{fig:dicoshapes}. 
All shapes share the same conductivity $\sigma=10$ and the same permittivity $\ve=1$,
except for the second ellipse which has the electromagnetic parameters $\sigma=5$ and $\ve=2$. The conductivity
and the permittivity of the background are $\sigma_0=1, \ve_0=0$.  To construct the shape
descriptors $\set{\mcl I(B_n)}_n$ of the dictionary, we set ${h_j}$ as in \eqref{eq:def_hj}
for four scales $j=-1,0,1,2$ and compute $\mN^j_{11}(t;B_n)$ in the frequency
domain via \eqref{eq:wMmnh_freq} and \eqref{eq:def_Mmn}, then followed by inverse Fourier
transform to go back to the time domain. 
\graphicspath{{figures/dico/}}
\def\figwidth{1.75cm}

\begin{figure}[htp]
  \centering
  \subfigure[\tiny{Circle}]{\includegraphics[width=\figwidth]{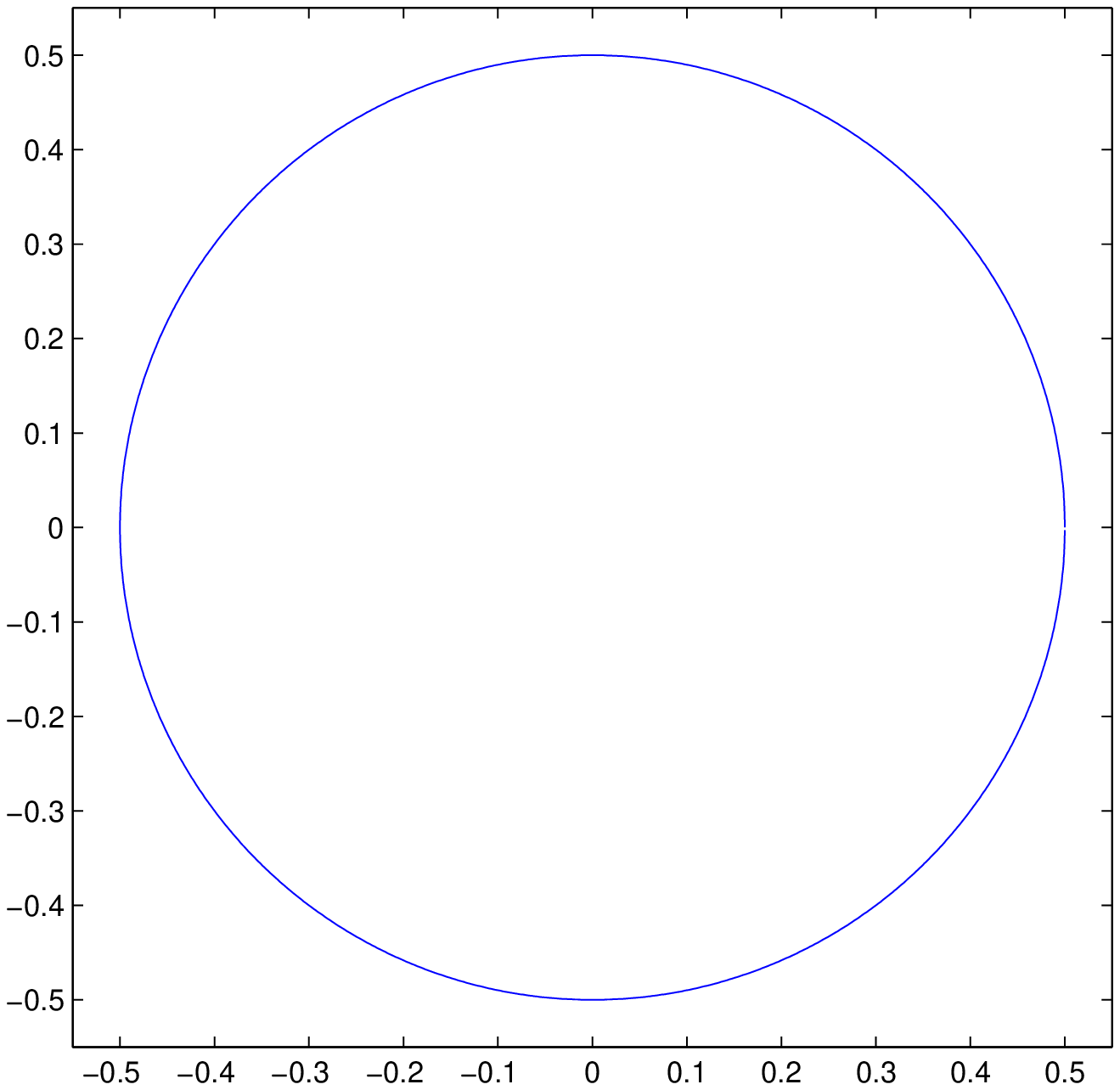}}
  \subfigure[\tiny{Ellipse}]{\includegraphics[width=\figwidth]{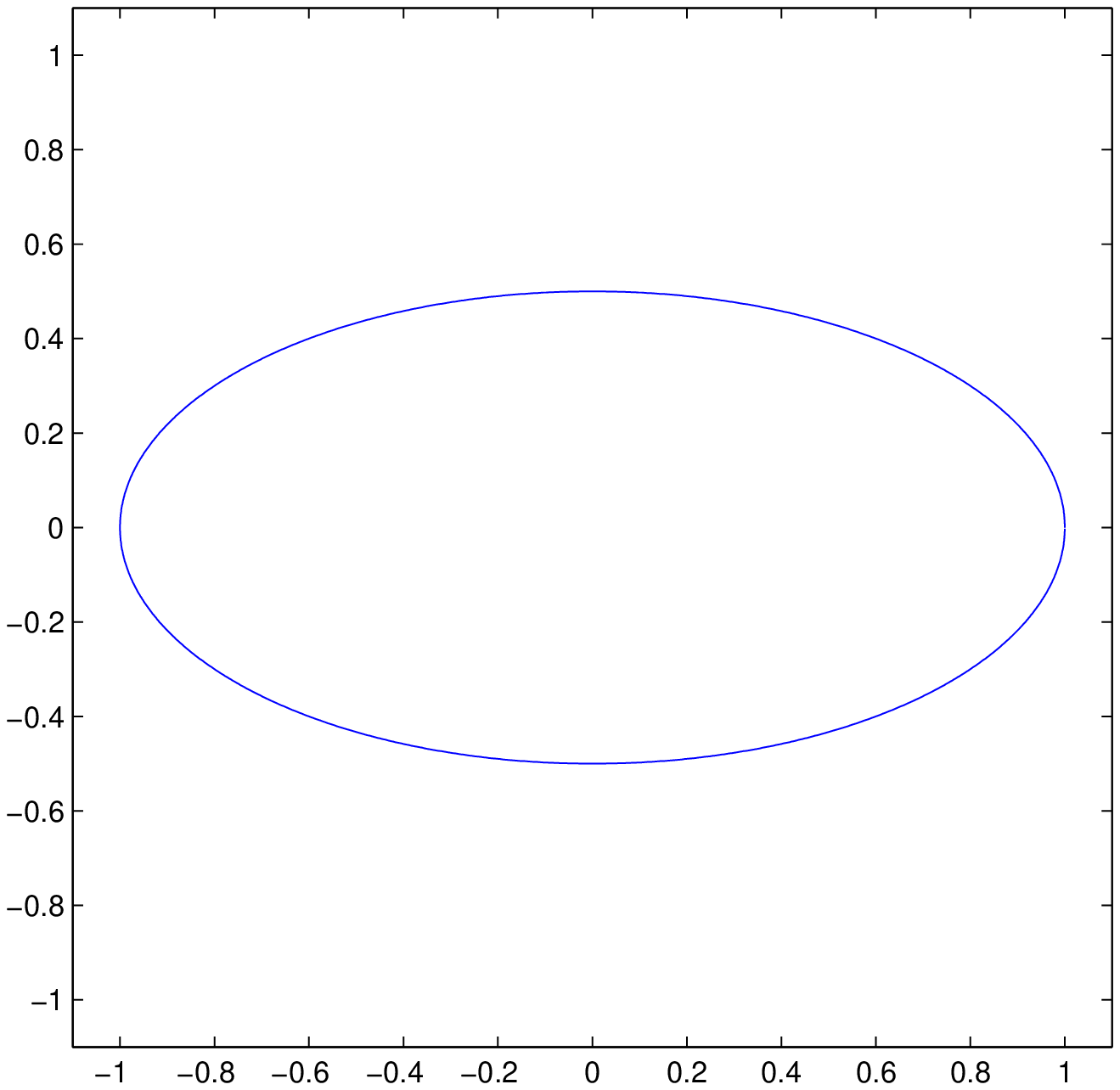}}
  \subfigure[\tiny{Flower}]{\includegraphics[width=\figwidth]{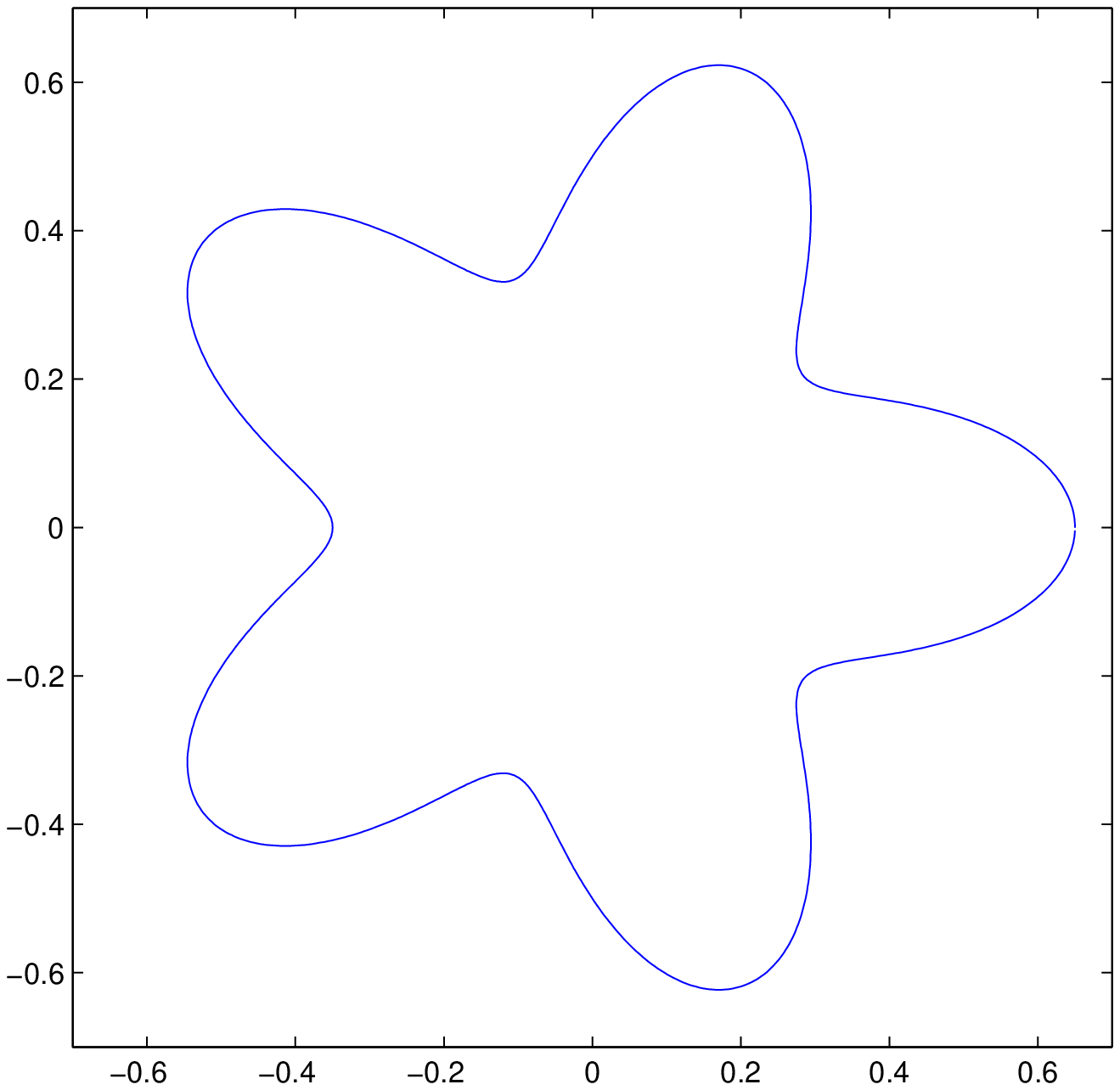}}
  \subfigure[\tiny{Square}]{\includegraphics[width=\figwidth]{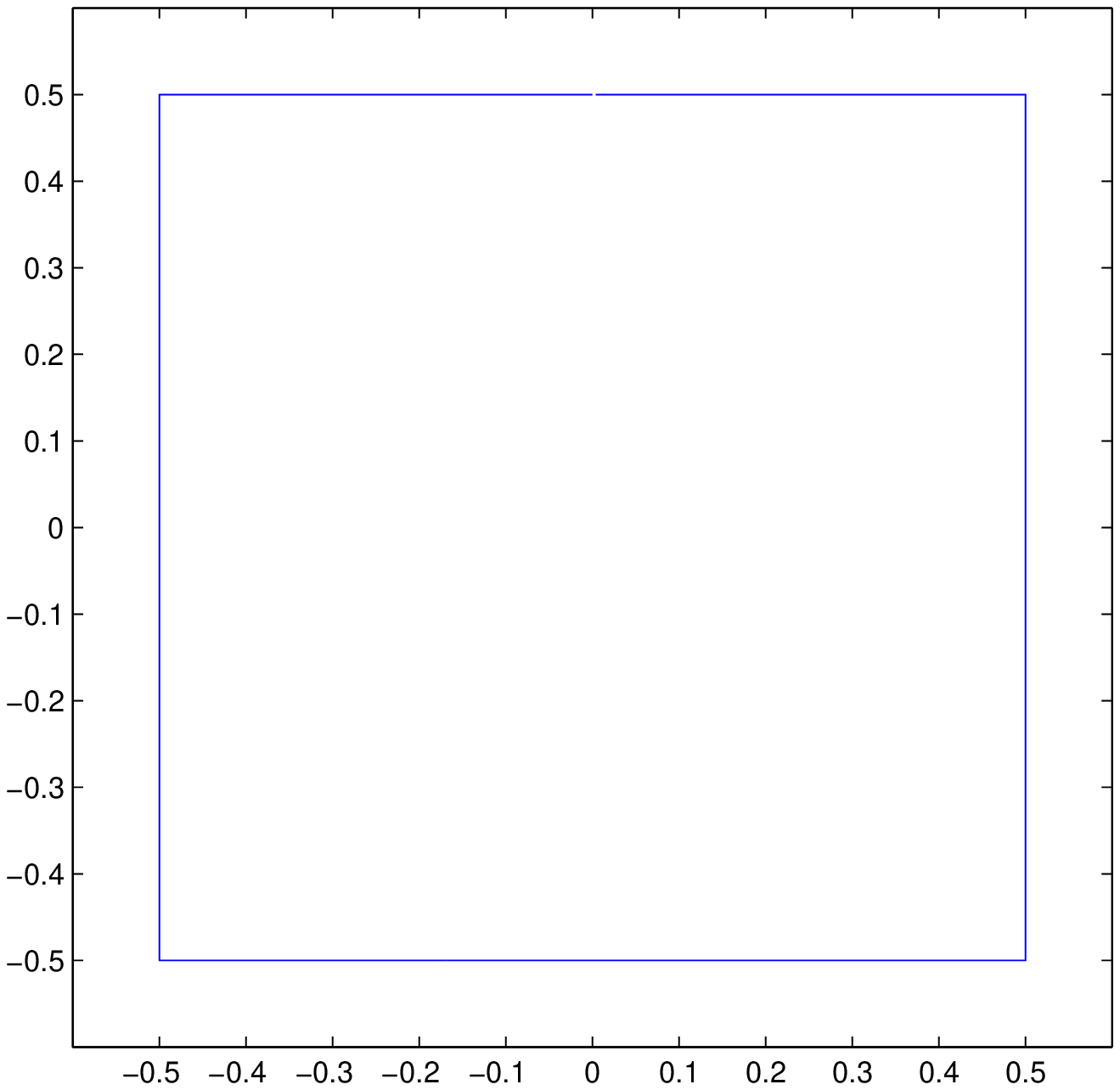}}
  \subfigure[\tiny{Rectangle}]{\includegraphics[width=\figwidth]{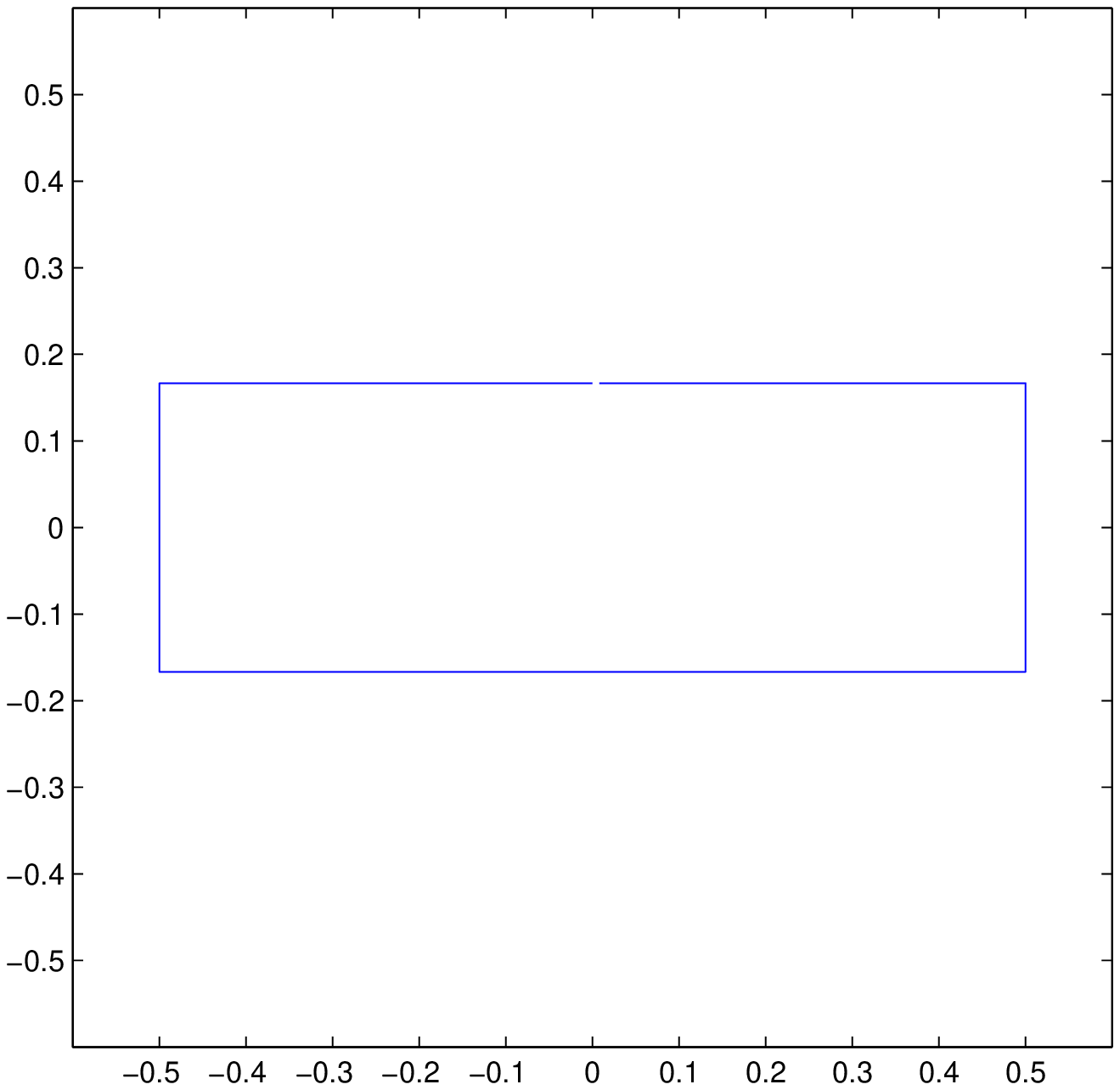}}
  \subfigure[\tiny{Letter A}]{\includegraphics[width=\figwidth]{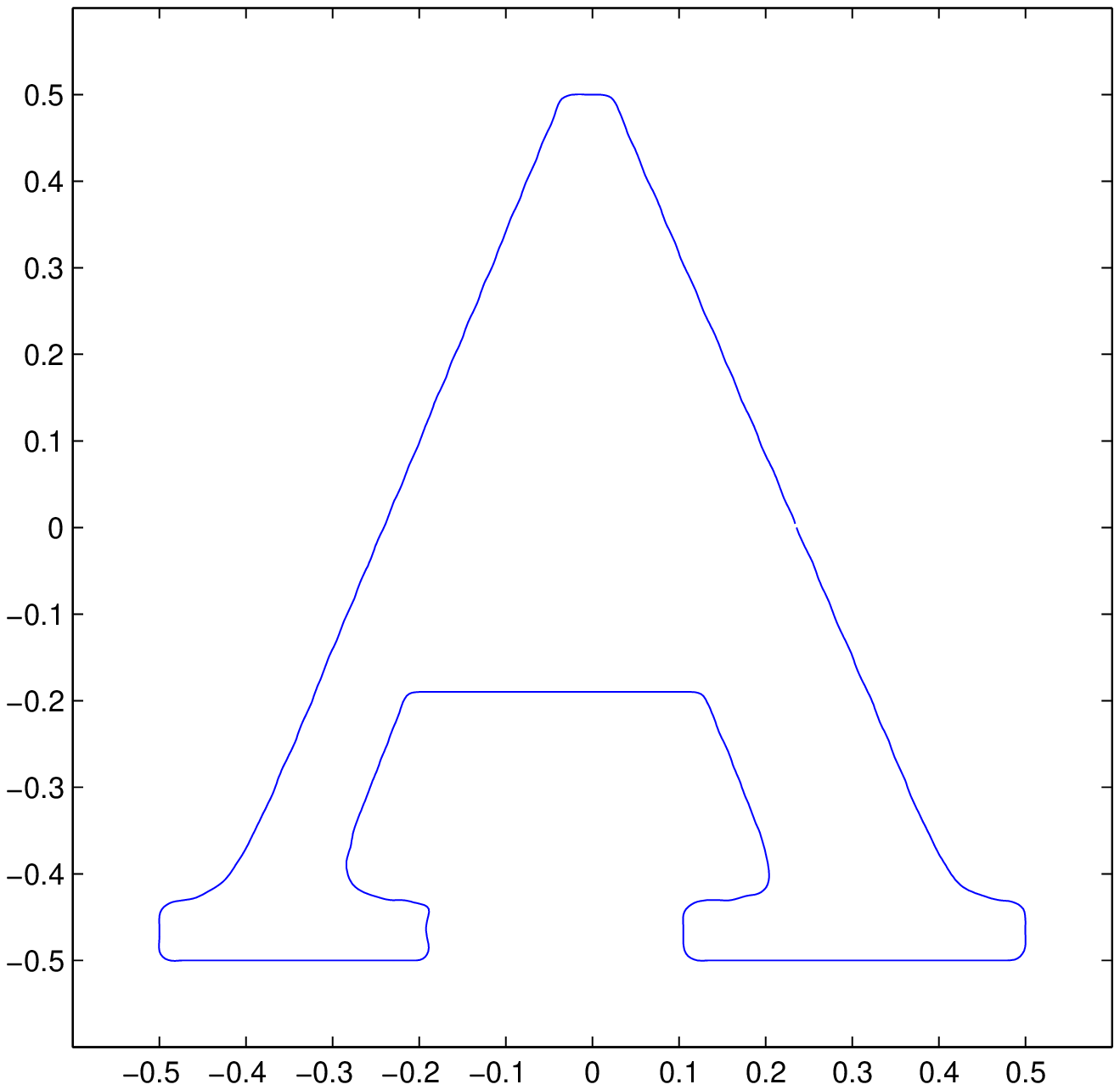}}
  \subfigure[\tiny{Letter L}]{\includegraphics[width=\figwidth]{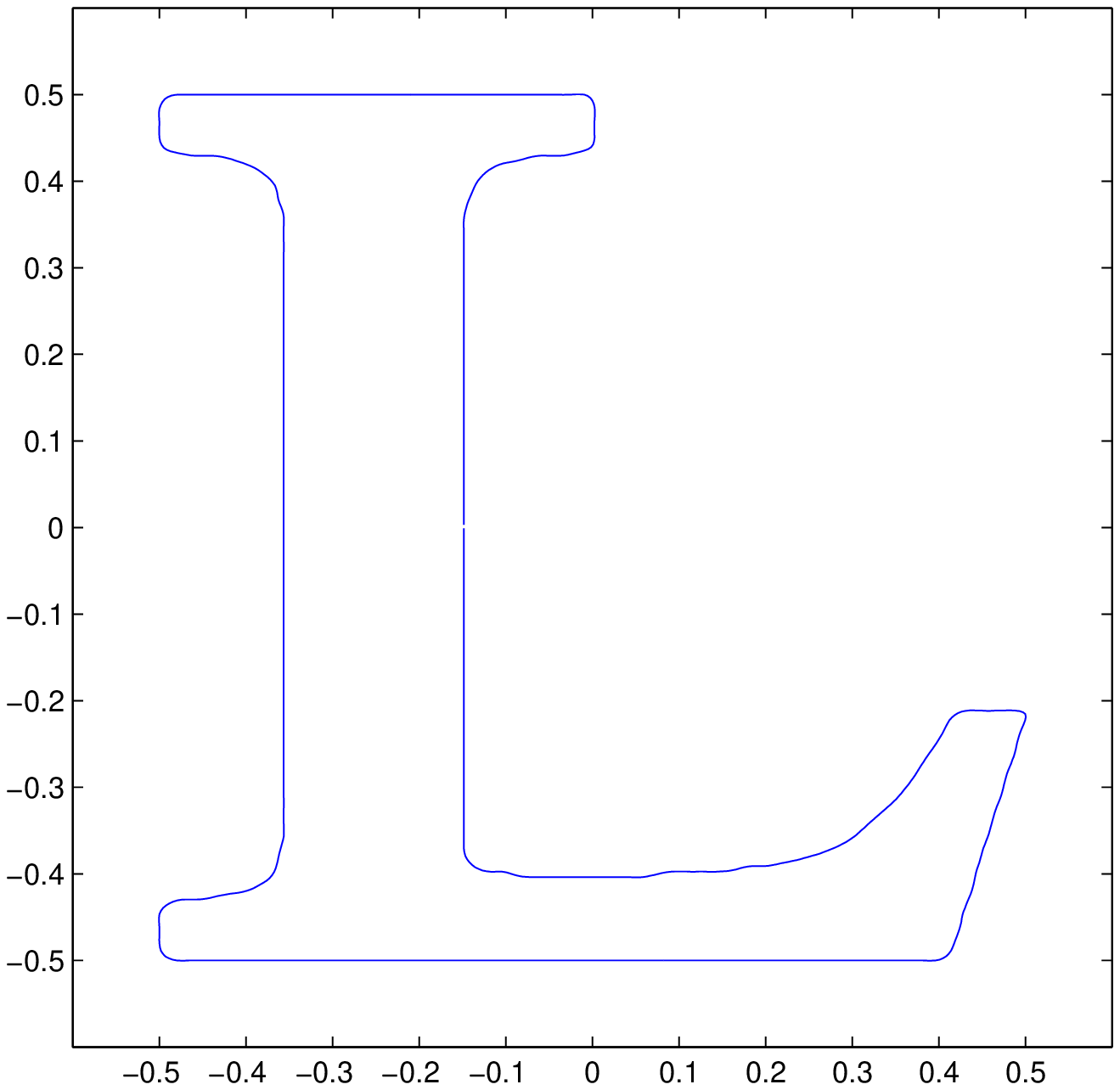}}
  \subfigure[\tiny{Ellipse 2}]{\includegraphics[width=\figwidth]{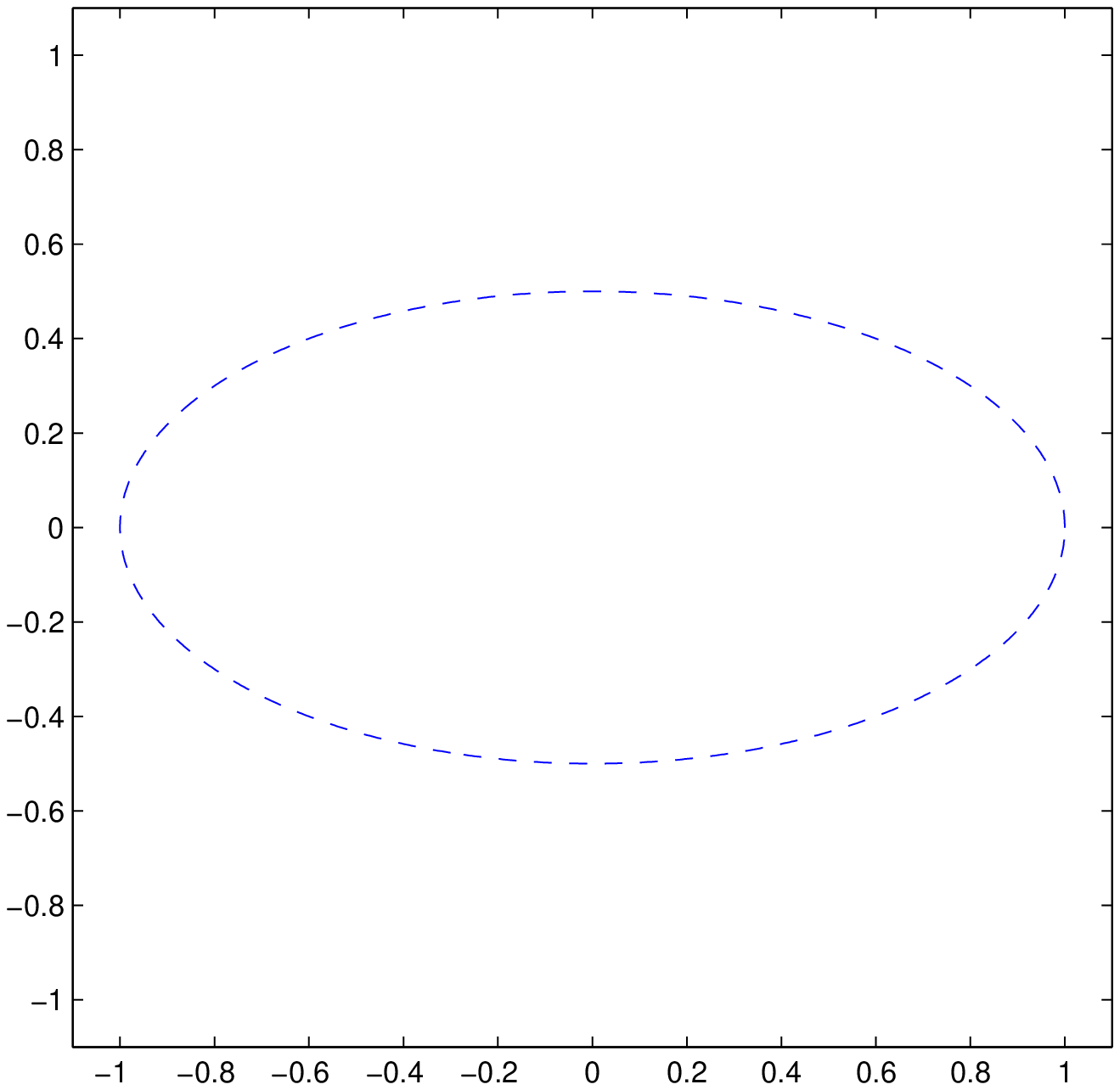}}
  \caption{A small dictionary of shapes. All the shapes have the same conductivity $\sigma=10$
    and the same permittivity $\vep=1$ except the ellipse in dot line which has
    $\sigma=5$ and  $\vep=2$.}
  \label{fig:dicoshapes}
\end{figure}

\paragraph{Data simulation}
\label{sec:simulation-data}
The same pulse shapes $h_j$ are used for the simulation of data.  The target $D$ is one
of the dictionary elements after applying the rotation $\theta=\pi/3$, the dilation $s=1.5$
and the translation $z=[0.1,0.1]^\top$ with $\top$ being the transpose. At the scale $j$, the MSR matrix denoted by
$\mV_j(t)$ is simulated on the time interval $[0,2^{-j}T]$ with $T=5$ using $N=2^9$
uniform samples, by evaluating the integral representation \eqref{eq:repr_solution}. More
specifically, we first obtain $\vphi(t)$ by solving \eqref{eq:repr_vphi_func_fulfill} (with
$h_j$ as the pulse shape) via the numerical scheme of Appendix \ref{sec:numer-solut-forw}. Then we
apply the single layer potential $\Sgl D$ on $\vphi(t)$. Further, each entry of the simulated matrix is
contaminated by some white noise following the normal distribution $\normallaw{\snoise^2}$
with
\begin{align*}
  \snoise = \frac {\snoiseper} {\sqrt{N_sN_r}} \Paren{\frac 1 {2^{-j}T}
    \int_0^{2^{-j}T} \norm{\mV_j(t)}^2_F \, dt}^{-1/2}
\end{align*}
with $\snoiseper$ being the percentage of the noise. 
Figure \ref{fig:MSR_noisy} shows the time profile of the entry $V_{11}$ in 
the MSR matrix for an
elliptical target simulated using the pulse shape $h_0$.

\graphicspath{{figures/}}
\def\figwidth{7cm}
\begin{figure}[htp]
  \centering 
  \subfigure[$\snoiseper=1$]{\includegraphics[width=\figwidth]{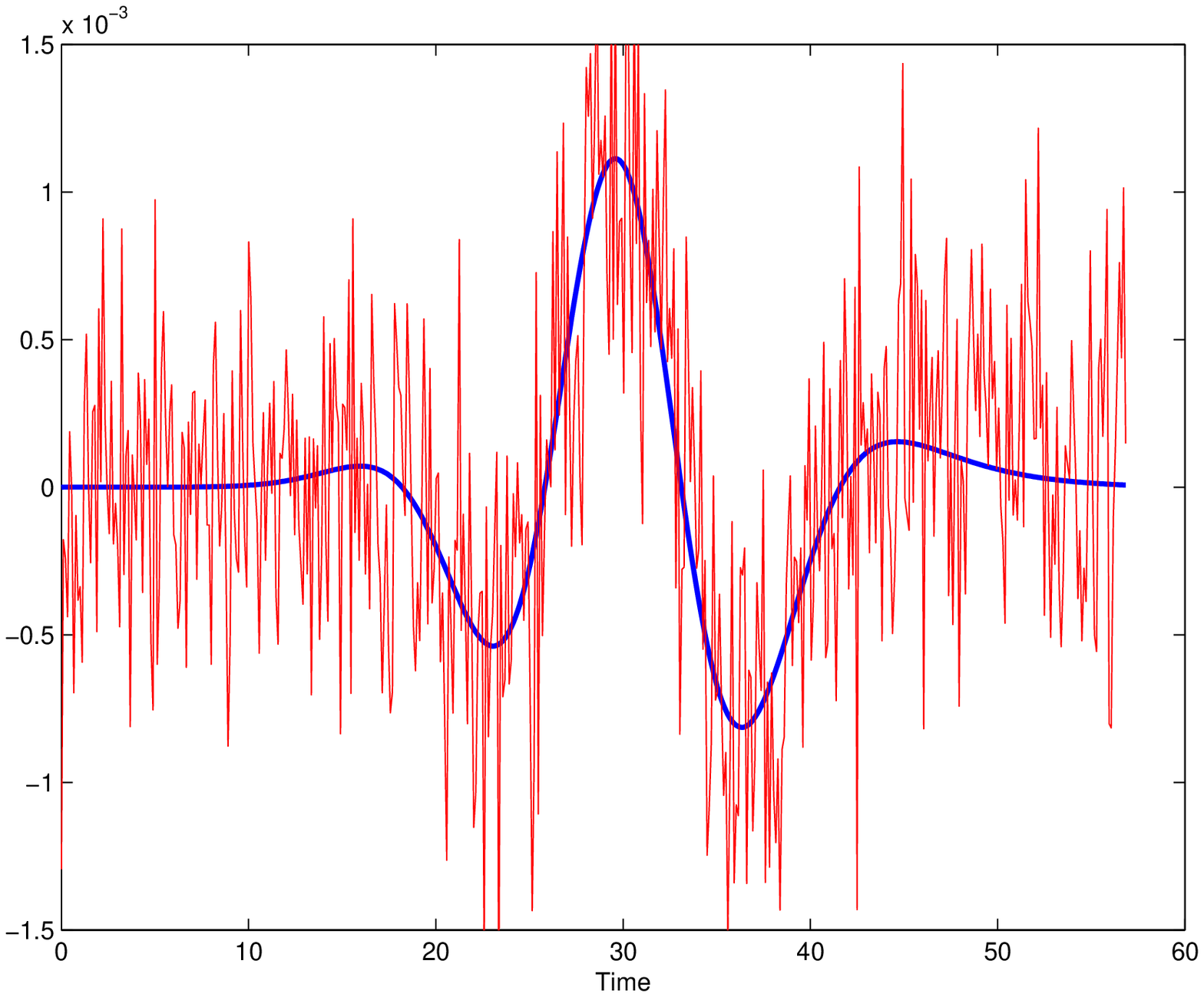}}
  \subfigure[$\snoiseper=2$]{\includegraphics[width=\figwidth]{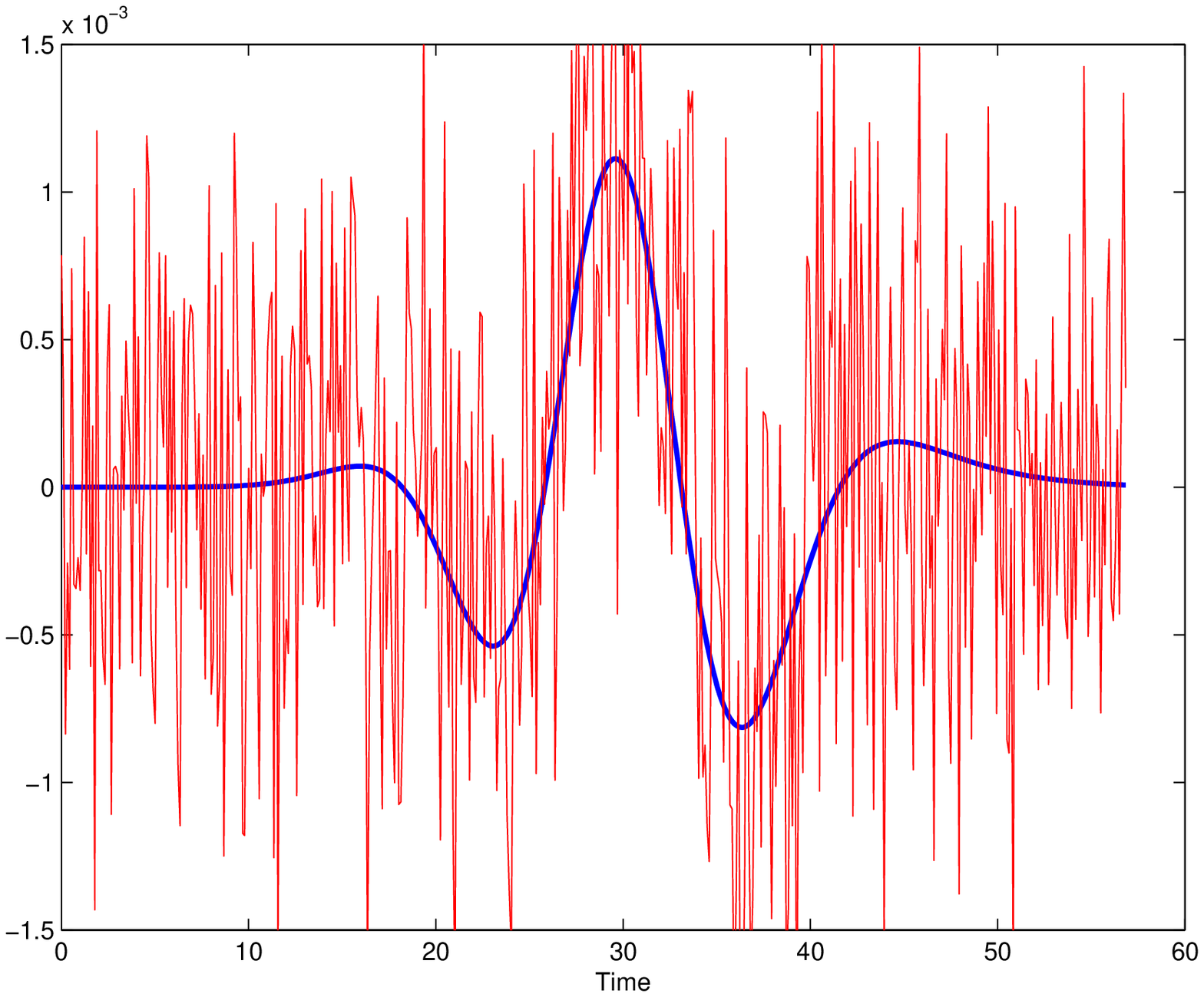}}
  \caption{Example of the MSR data $V_{11}(t)$ corresponding to an ellipse. In blue: without noise. In red:
    with (a) $100\%$ and (b) $200\%$ of noise.}
  \label{fig:MSR_noisy}
\end{figure}

\paragraph{Shape identification}
\label{sec:shape-recognition}
For each scale $j$ we reconstruct the filtered polarization tensor $\mN^j_{11}(t)$ from
the simulated data by inverting the linear system \eqref{eq:Vsr_linsys_time} in the time
domain (the operator $\mL$ is constructed as in~\eqref{eq:GPT_expansion} with the
truncation order $K=1$). Furthermore, the symmetry of $\mN^j_{11}(t)$ is incorporated as a
constraint in the inversion in order to enhance the robustness. The shape descriptors are then
computed via \eqref{eq:def_Ijn}. Finally the euclidean norm $$\ve(D,B_n) = \norm{\mcl I(D)
  - \mcl I(B_n)}$$
is evaluated for the whole dictionary and the shape is identified as the one yielding the
smallest value.

\subsection{Results of identification}
\label{sec:results}

For each shape of the dictionary, we simulate data and identify it using the procedure
described above. Figure \ref{fig:errorbar1} shows the results of shape identification for
a limited view configuration with the aperture $\alpha=\pi/16$ at two noise levels
$\snoiseper=100\%$ and $200\%$. The error $\ve(D,B_n)$ is represented here by error bars,
where the $m$-th bar in the $n$-th group corresponds to $\ve(D,B_m)$ of the identification
experiment with the shape $D$ generated by $B_n$ (labeled by its name). The shortest bar
in each group is the identified shape and is marked in green, while the true shape is
marked in red in case that the identification fails. Each error bar is the average of the
same experiment with 100 independent realization of white noise.  It can be seen that the
identification succeeded for all shapes with $100\%$ of noise, and it failed only for the
circle with $200\%$ of noise.

\begin{figure}[htp]
  \centering
  \subfigure[$\snoiseper=1$]{\includegraphics[width=\figwidth]{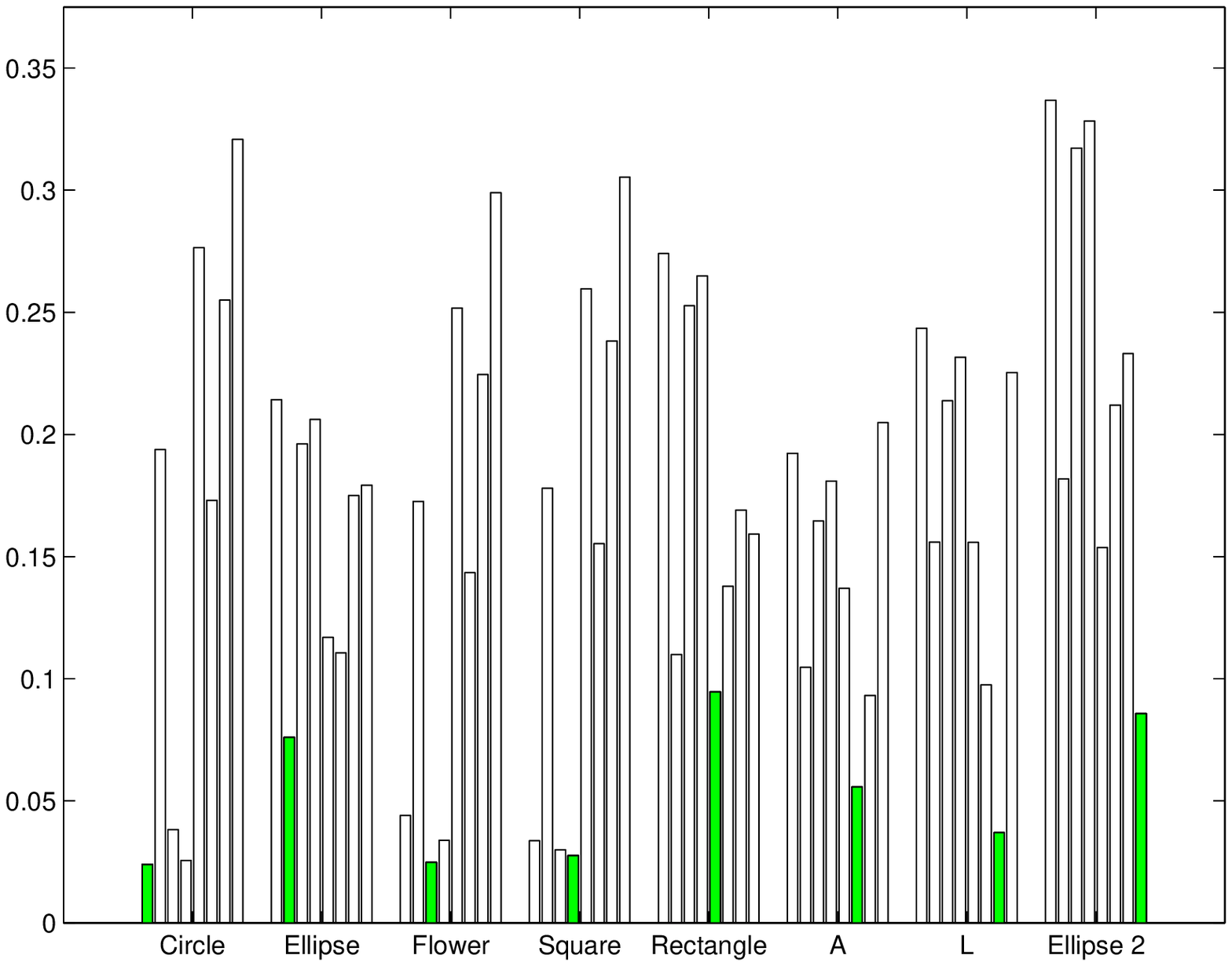}}
  \subfigure[$\snoiseper=2$]{\includegraphics[width=\figwidth]{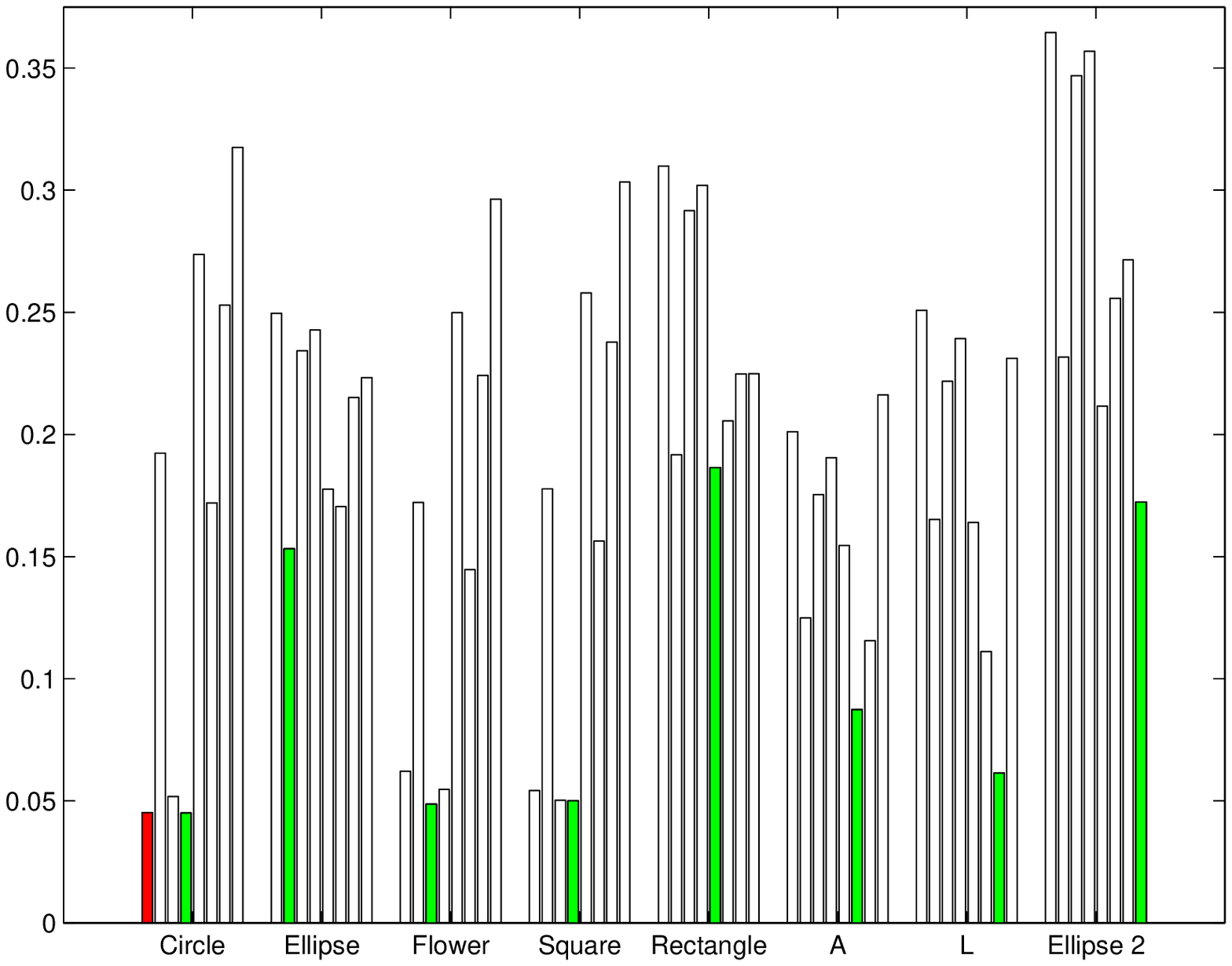}}
  \caption{Results of identification at two noise levels using a limited view configuration
    with the aperture $\alpha=\pi/16$.}
  \label{fig:errorbar1}  
\end{figure}


\paragraph{Robustness}
\label{sec:robustn-against-nois}

Figure \ref{fig:robustness_noise} illustrates the robustness of the proposed method in a
noisy environment for two settings of limited view with the aperture $\alpha=\pi/8$ and
$\alpha=\pi/32$. Each curve represents the probability of successful identification as a
function of $\snoiseper$ which ranges from $25\%$ to $800\%$, obtained by repeating at
every noise level the experiment 1000 times with independent realizations of white
noise. The horizontal line at 0.125 marks the threshold that the proposed matching method
performs better than a random guess. It can be seen that the angle of view can affect the
performance, and in both cases all shapes are correctly identified with $100\%$ of
noise. It is worth noticing that certain shapes, like the letters and the flower, exhibits
an extraordinary robustness.

\begin{figure}[htp]
  \centering
  \subfigure[$\alpha=\pi/8$]{\includegraphics[width=\figwidth]{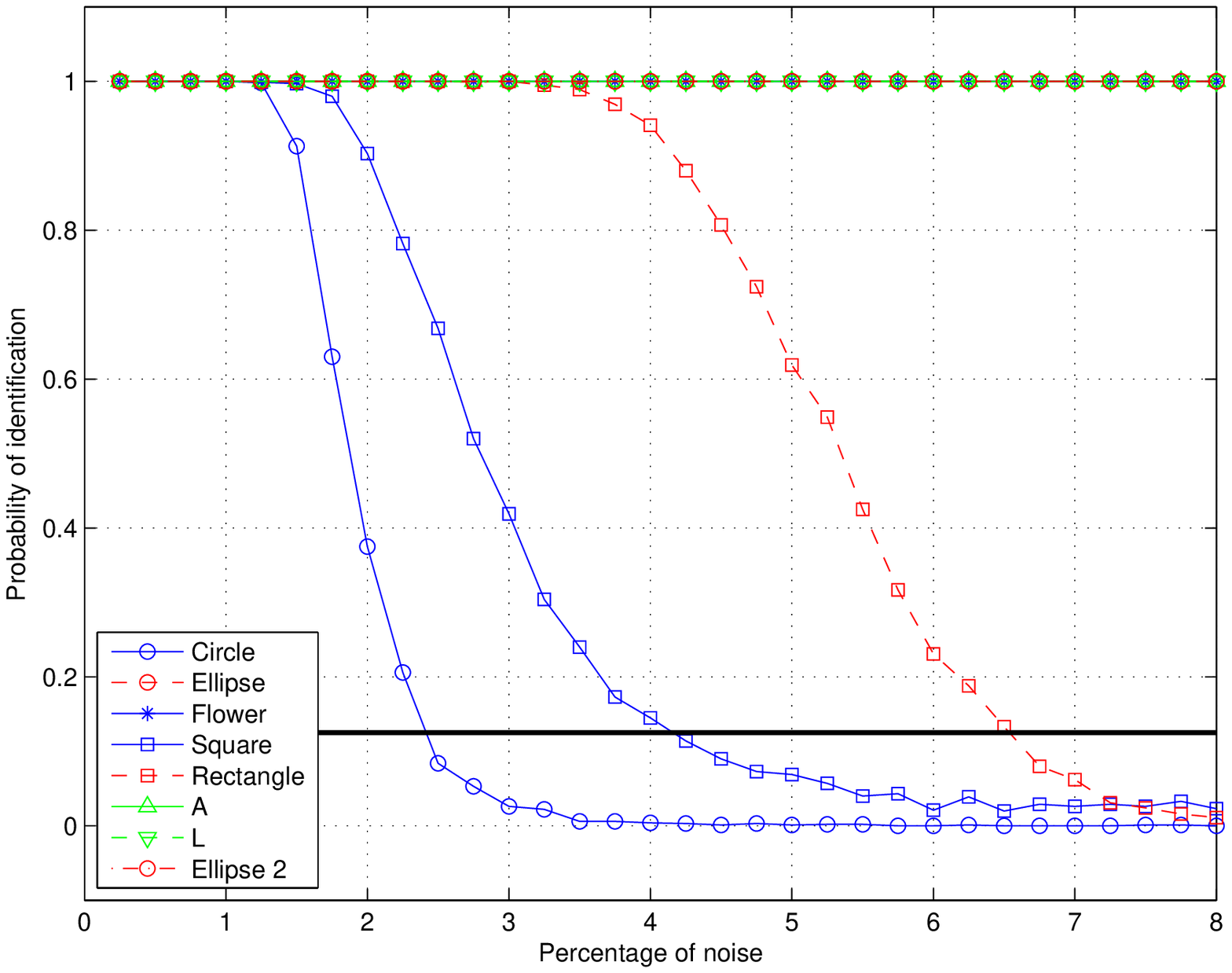}}
  \subfigure[$\alpha=\pi/32$]{\includegraphics[width=\figwidth]{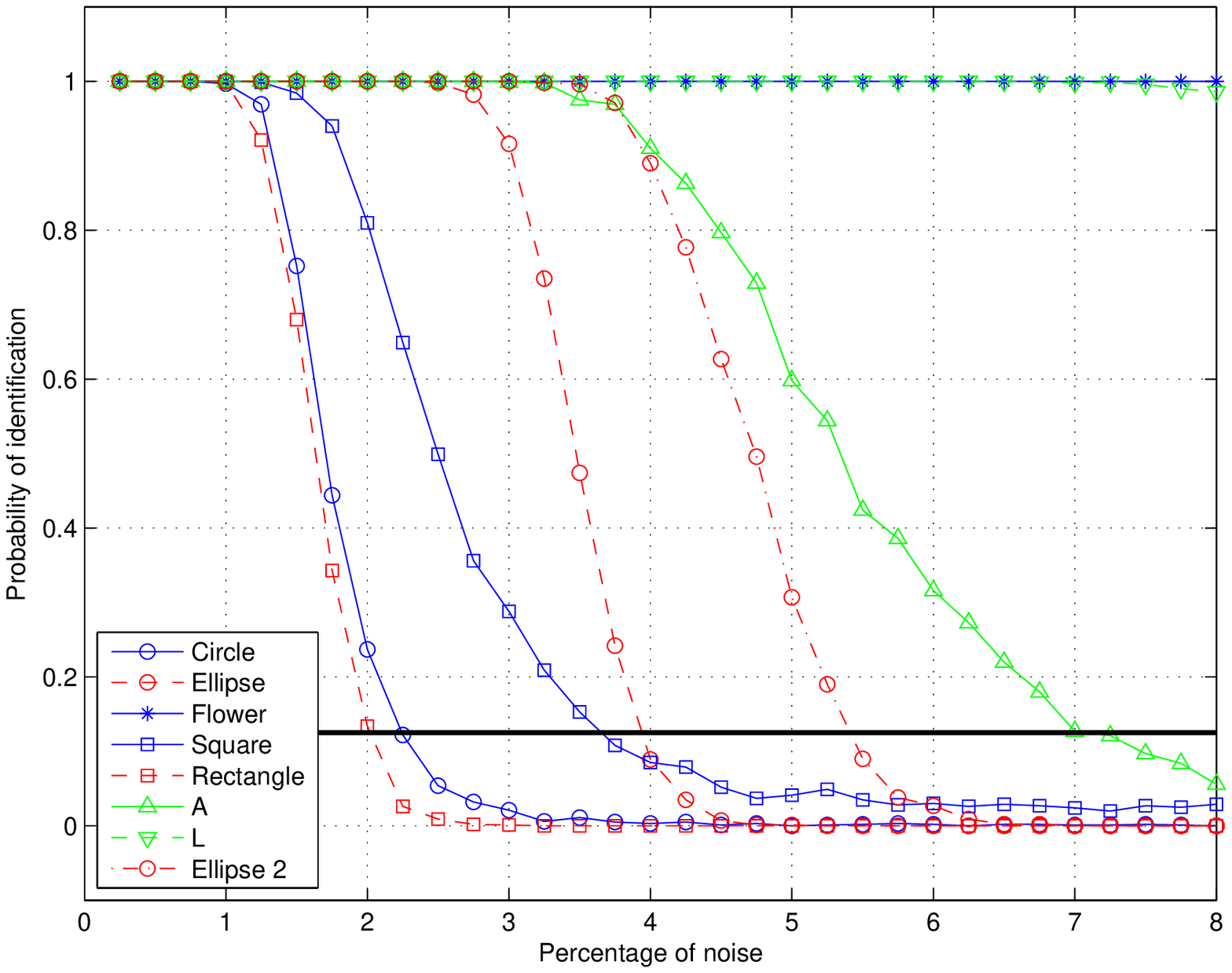}}
  \caption{Results of identification at various noise levels with the shape descriptors of
    4 scales and the aperture (a): $\alpha=\pi/8$ and (b): $\alpha=\pi/32$.}
  \label{fig:robustness_noise}  
\end{figure}

\paragraph{Number of scales}
\label{sec:number-scales}
The number of scales has an important impact on the robustness of the identification. A
large number of scales contains more information hence gives a better performance of
identification. On the contrary, the overall performance is reduced when the number of
scales is insufficient. This can be seen from Figure \ref{fig:robustness_scale} where the
same experiment in Figure \ref{fig:robustness_noise} is carried out with the scales $j=-1$
and $j=-1,0$ respectively.

\begin{figure}[htp]
  \centering
  \subfigure[$1$ scale]{\includegraphics[width=\figwidth]{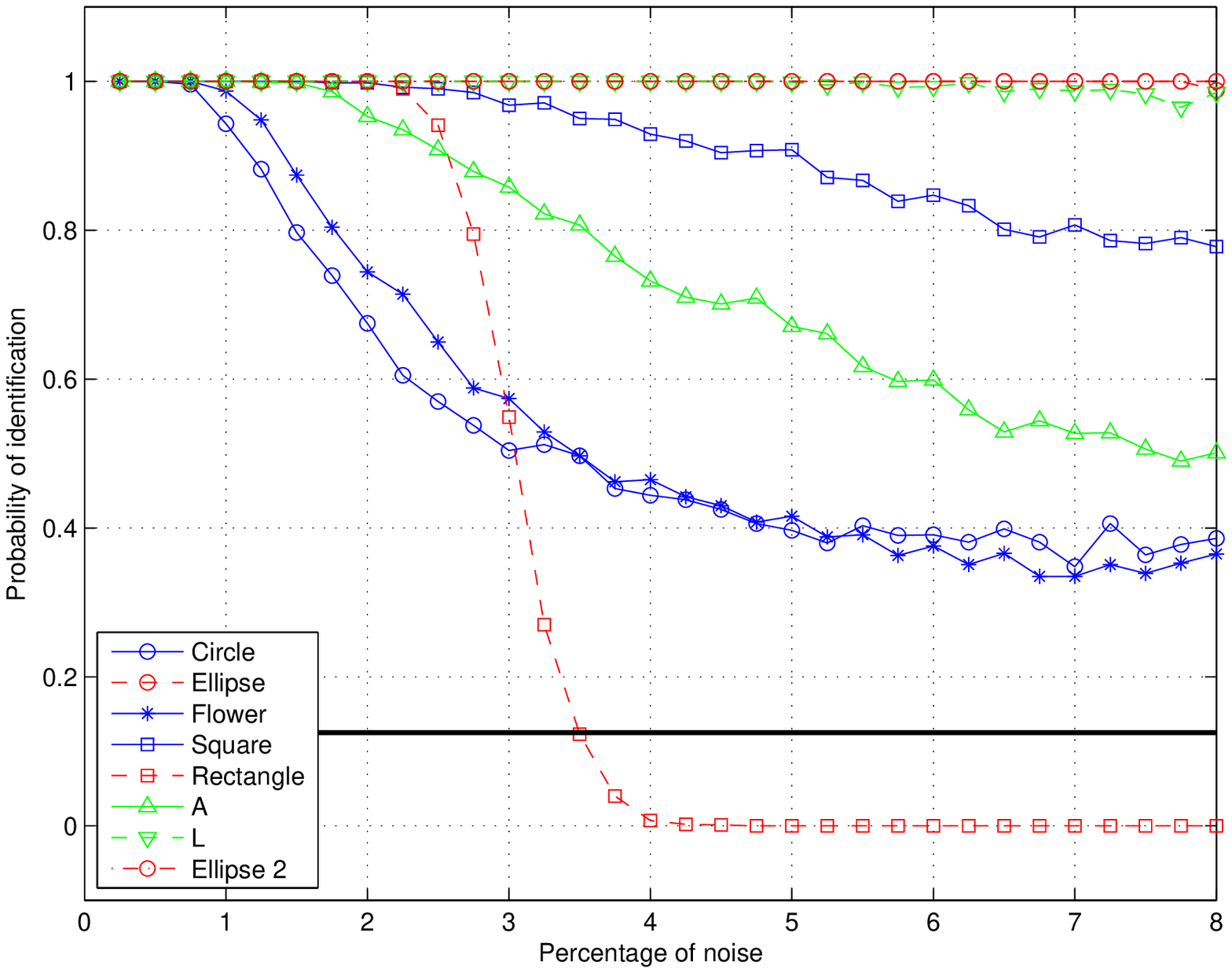}}
  \subfigure[$2$ scales]{\includegraphics[width=\figwidth]{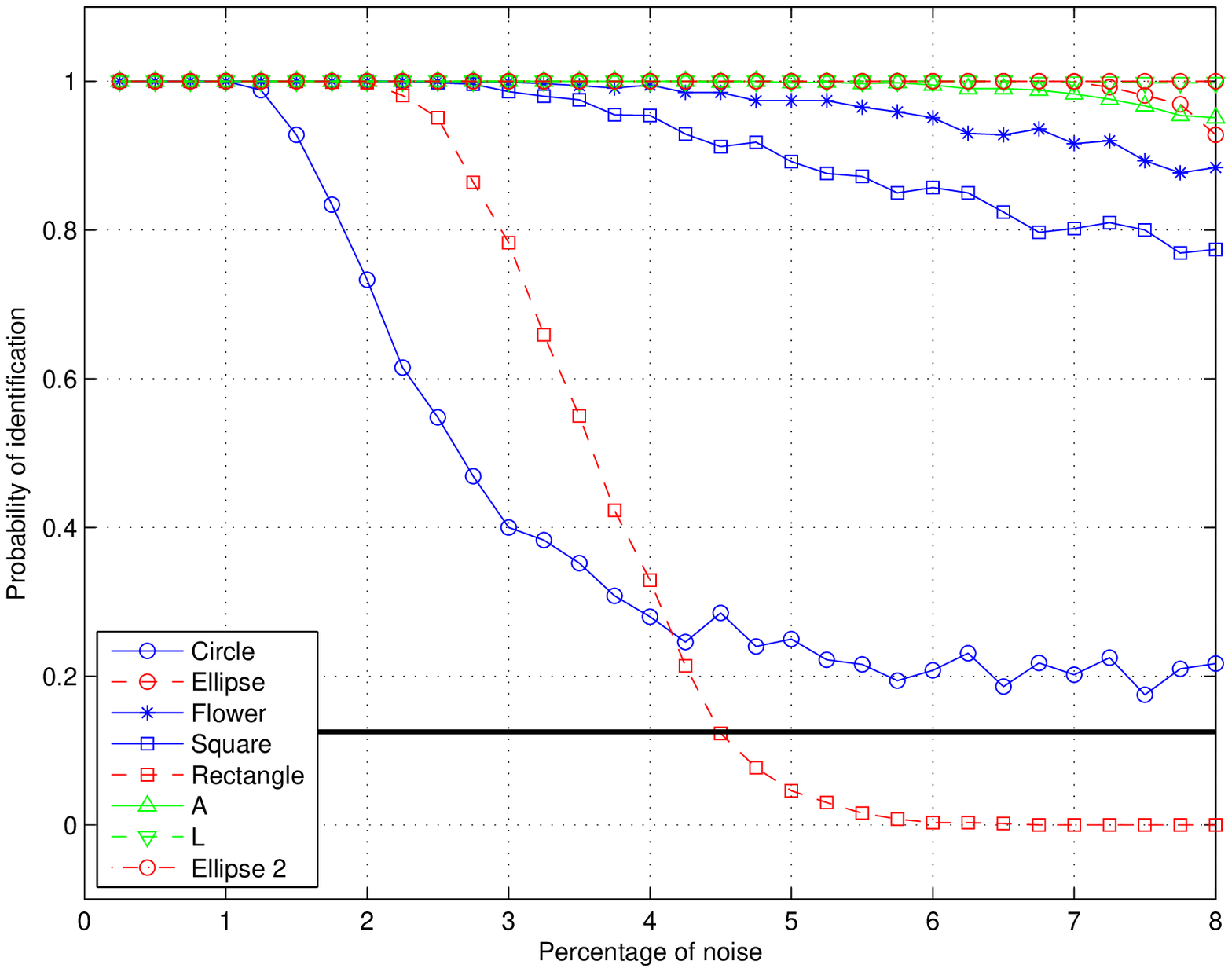}}
  \caption{Same experiment as in Figure \ref{fig:robustness_noise} (a) with the scales (a)
    $j=-1$ and (b) $j=-1,0$ only.}
  \label{fig:robustness_scale}
\end{figure}



\section{Discussion and conclusion}
\label{sec:conclusion}

In this paper we presented a new time domain multi-scale method of shape identification
for electro-sensing using pulse-type signals. The method is based on transform-invariant
shape descriptors which are computed from the filtered polarization tensor at
multi-scales, and enjoys a remarkable robustness even in a highly noisy environment with
far field transmitters of very small angle of view. Time domain data, acquired using
pulses of different scales, contain information about the target at different frequency
bands, and allow a better distinction of shapes than using a single scale. The new method
improves also the results of the multi-frequency approach proposed in
\cite{ammari_shape_2014}. We reported here only results on conductive objects ($\sigma\gg
\sigma_0$, compared to the surrounding water), while a similar performance can also be
observed on resistive objects ($\sigma\ll \sigma_0$) and in this case one needs to adapt
the range of the scales 
to the new physical values in order to obtain good distinguishability between shapes.  The
new method can also be generalized to the modeling of electric fish in
\cite{ammari_modeling_2013} and this will be the subject of a forthcoming paper. We also
plan to optimize the pulse shape for a given dictionary of targets. Finally, it is
expected that the proposed time-domain multi-scale algorithm can be extended to shape
identification and classification in echolocation \cite{han_echo} and in imaging from
induction data \cite{junqing1, junqing2}.



\appendix

\section{Numerical solution of the forward problem}
\label{sec:numer-solut-forw}

We aim to simulate the perturbation $u(t,x)-U(t,x)$ for $t\in [0, T]$ using the representation
\eqref{eq:repr_solution}. We will solve the system \eqref{eq:repr_vphi_func_fulfill}
on the time interval $[0, T]$ under the initial condition $\vphi(0)=0$ (since $\vphi$ is causal)
by combining a boundary element method (BEM) in space and a finite difference scheme in
time.

The time interval $[0,T]$ is equally divided into $N$ parts with the time step $\Dt=T/N$ and we
denote by $\vphi^n(x)=\vphi(n\Dt,x)$ for $n=0\ldots N$, so that it holds approximately
\begin{equation*}
  \vphi'(n\Dt, x) \simeq \frac{\vphi^{n}(x) -\vphi^{n-1}(x)}{\Dt} \ \text{ for a.e. } x\in\p D.
\end{equation*}
The same discretization in time is applied to  term on the right-hand side, $(1+\alpha\p_t)\frac{\partial  U}{\partial \nu}$, and
we write $b=\frac{\partial U}{\partial \nu}$. Inserting these into \eqref{eq:repr_vphi_func_fulfill} and after some
simple manipulations, we get
\begin{align}
  \label{eq:psin_dtime}
  \Paren{\tilde\lambda I - \Kstar D}\Brack{\vphi^n}
  = b^n + {\frac {\alpha}{\Dt+\alpha}} \Paren{\hKstarf D {\vphi^{n-1}} - b^{n-1}}
\end{align}
with $\tilde \lambda = \frac{\e/\Dt + \sgm + 1}{2(\e/\Dt + \sgm -1)}$, and the operator
$(\tilde \lambda I - \Kstar D)$ is clearly invertible on $\LtpD$. In the space domain (with the
time being fixed), $\Pz$ elements are used for the discretization of $L^2(\p D)$ function. Let
$x(\theta)$ be the parameterization of the boundary $\p D$ with $\theta\in[0,1]$. We denote by
$\vphi^n_j=\vphi(n\Dt, x(\theta_j))$ the $j$-th coefficient of $\vphi(n\Dt)$ under the $\Pz$
basis, and by $\Albd,\Ahlf$ the matrix representation of $(\tilde \lambda I - \Kstar D),
\hKstar D$ under $\Pz\times\Pz$ basis. Denoting by $\bphi^n=(\vphi^n_j)_j, \bb^n=(b^n_j)_j$ the
discrete coefficient vector, finally the time-space discretization yields the following linear
system for $n=1\ldots N$:
\begin{align}
  \label{eq:psinj_dtime_dspace}
  \Albd \bphi^n = \bb^n + {\frac {\alpha}{\Dt+\alpha}} \Paren{\Ahlf \bphi^{n-1} - \bb^{n-1}}
\end{align}
with the initial state $\bphi^0 = 0$. Then \eqref{eq:psinj_dtime_dspace} is inverted
iteratively for $n=1\ldots N$ and we inject the solution $\set{\bphi^0,\ldots,\bphi^n}$ into
\eqref{eq:repr_solution} to get the desired data by evaluating the single layer potential.


\bibliographystyle{plain}
\bibliography{Biblio}

\end{document}